\documentclass[a4paper]{amsart}
\usepackage{color}
\usepackage{tikz}
\usepackage{amssymb,amsmath,amsthm,amstext,amsfonts}
\usepackage{psfrag}
\usepackage{url}
\usepackage{amsfonts}
\usepackage{amsmath}
\usepackage{mathrsfs}
\usepackage{graphicx}
\usepackage{amsmath,amsthm,amssymb,amscd}
\usepackage{enumerate}
\usepackage{circledsteps}
\usepackage[colorlinks,linkcolor=black,anchorcolor=black,citecolor=black,hyperindex=true,CJKbookmarks=true]{hyperref}

\usepackage{epsfig}

\makeatletter \@addtoreset{equation}{section} \makeatother

\renewcommand\thetable{\thesection.\@arabic\c@table}

\theoremstyle{plain}
\newtheorem{maintheorem}{Theorem}

\newtheorem{theorem}{Theorem}[section]
\newtheorem{proposition}{Proposition}[section]
\newtheorem{lemma}{Lemma}[section]
\newtheorem{corollary}{Corollary}[section]
\newtheorem{definition}{Definition}[section]
\newtheorem{remark}{Remark}[section]

\theoremstyle{remark}

\long\def\begcom#1\endcom{}

\newcommand{\length}{\operatorname{\length}}

\def\length{\operatorname{length}}

\newcommand{\bl} {\begin{lemma}}
\newcommand{\el} {\end{lemma}}
\newcommand{\bt} {\begin{theorem}}
\newcommand{\et} {\end{theorem}}

\newcommand{\bp}{\begin{proof}}
\newcommand{\ep}{\end{proof}}
\newcommand  {\ee} {\end{equation}}
\newcommand  {\beq} {\begin{eqnarray*}}
\newcommand  {\eeq} {\end{eqnarray*}}

\newcommand  {\bd} {\begin{definition}}
\newcommand  {\ed} {\end{definition}}

\newcommand{\cM}{\mathcal{M}}
\newcommand{\cB}{\mathcal{B}}

\def\ep{\noindent{\hfill $\Box$}}

\makeatletter
\@namedef{subjclassname@2020}{2020 Mathematics Subject Classification}
\makeatother

\begin{document}

\title{Ergodic average of typical orbits and typical functions}

\author{Xiaobo Hou, Wanshan Lin and Xueting Tian}

\address{Xiaobo Hou, School of Mathematical Sciences,  Fudan University\\Shanghai 200433, People's Republic of China}
\email{20110180003@fudan.edu.cn}

\address{Wanshan Lin, School of Mathematical Sciences,  Fudan University\\Shanghai 200433, People's Republic of China}
\email{21110180014@m.fudan.edu.cn}

\address{Xueting Tian, School of Mathematical Sciences,  Fudan University\\Shanghai 200433, People's Republic of China}
\email{xuetingtian@fudan.edu.cn}


\begin{abstract}
In this article we mainly aim to know what kind of asymptotic behavior of typical orbits can display. For example, we show in any transitive system, the emprical measures of a typical orbit can cover all emprical measures of dense orbits and can intersect some physical-like measures. In particular, if the union set of emprical measures of all dense orbits is not singleton, then the typical orbit will display historic behavior simultaneously for typical continuous functions and the limit set of ergodic average along every continuous function equals to a closed interval composed by the union of limit sets of ergodic average on all dense orbits. Moreover, if the union set of emprical measures of all dense orbits contains all ergodic measures, the above interval equals to the rotation set.     These results are not only suitable for systems with specification-like properties or minimal systems, but also suitable for many other systems including all general (not assumed uniformly hyperbolic) nontrivial homoclinic classes and Bowen eyes. Moreover, we introduce a new property called $m$-$g$-product property weaker than classical specification property and minimal property and nontrivial examples are constructed.
\end{abstract}

\maketitle

\section{Introduction}

Throughout this paper, let $(X,d)$ be a compact metric space, and $T:X \rightarrow X$ be a continuous map. Such $(X,T)$ is called a dynamical system. Let $\cM(X)$, $\cM(X,T)$, $\cM^{e}(X,T)$ denote the spaces of probability measures, $T$-invariant, $T$-ergodic probability measures, respectively. Let $\mathbb{N}$, $\mathbb{N^{+}}$ denote non-negative integers, positive integers, respectively. Let $C(X)$ denote the space of real continuous functions on $X$ with the norm $\|f\|:=\sup\limits_{x\in X}|f(x)|.$ In a Baire space, a set is said to be residual if it has a dense $G_\delta$ subset. 

In this article we study ergodic average of typical orbits. There are many kinds of asymptotic behaviors related to ergodic average, such as irregular points, saturated sets and physical-like measures. Now we recall definitions of these concepts.
\begin{itemize}
	\item[$\bullet$] For any $x\in X$ and any $f\in C(X),$ let $$\underline{f}(x)=\liminf\limits_{n\to+\infty}\frac1n\sum_{i=0}^{n-1}f(T^ix)\text{ and } \overline{f}(x)=\limsup\limits_{n\to+\infty}\frac1n\sum_{i=0}^{n-1}f(T^ix).$$ For any $f\in C(X)$,  define the {$f$-irregular set} as $I(f,T) := \left\{x\in X\mid \underline{f}(x)\neq \overline{f}(x)\right\}.$
	The $f$-irregular set and the irregular set, the union of $I(f,T)$ over all continuous functions of $f$ (denoted by $IR(T)$), arise in the context of multifractal analysis and have been studied a lot,
	for example,  see \cite{Dowker1953,PP1984, BS2000,  Pesin1997, CKS2005, Thompson2010, DOT}.  The irregular points  are also called points with historic behavior,
	see \cite{Ruelle2001, Takens2008}. Note that $$I(f,T)=\cup_{k\in\mathbb{N^{+}}}\cap_{N\in\mathbb{N^{+}}}\cup_{n\geq N}\cup_{m\geq N}\left\{x\in X\mid |\frac1n\sum_{i=0}^{n-1}f(T^ix)-\frac1m\sum_{i=0}^{m-1}f(T^ix)|\geq\frac{1}{k} \right\},$$ and $IR(T)=\cup_{f\in C(X)}I(f,T)=\cup_{i\in\mathbb{N^{+}}}I(f_i,T)$ where $\{f_i\}_{i\in\mathbb{N^+}}$ is a dense subset of $C(X).$ So $I(f,T)$ and $IR(T)$ are both Borel sets.
	From Birkhoff's ergodic theorem, the irregular set is not detectable from the point of view of any invariant measure.
	However, the irregular set may have strong dynamical complexity in sense of Hausdorff dimension, Lebesgue positive measure, topological entropy, topological pressure, distributional chaos and residual property etc.
	Pesin and Pitskel \cite{PP1984} are the first to notice the phenomenon of the irregular set
	carrying full topological entropy in the case of the full shift on two symbols.   There are lots of advanced results to show that the irregular points can carry full entropy in symbolic systems, hyperbolic systems, non-uniformly expanding or hyperbolic systems,  and systems with specification-like or shadowing-like properties, for example, see \cite{BS2000, Pesin1997, CKS2005, Thompson2010, DOT, LLST2017, TV2017}. For topological pressure case see \cite{Thompson2010}, for Lebesgue positive measure see \cite{Takens2008, KS2017}, for distributional chaos see \cite{CT2021} and for residual property see \cite{DTY2015,LW2014}.
	\item By \cite[Theorem 6.4]{Walters1982} there exists a countable  set of continuous functions $\{f_i\}_{i\in\mathbb{N^+}}$  such that $||f_i||=1$ for any $i\in\mathbb{N^+}$ and  $$\rho(\mu_{1},\mu_{2}):=\sum_{i=1}^{+\infty}\frac{|\int f_id\mu_{1}-\int f_id\mu_{2}|}{2^{i}}$$ defines a metric for the weak*-topology on $\mathcal{M}(X).$ 
	Given a point $x\in X$, let $V_T(x)$ be the set of accumulation points of the empirical measure of the Birkhoff average $\frac{1}{n}\sum_{i=0}^{n-1}\delta_{T^{i}x}$, where $\delta_{y}$ is the Dirac measure at point $y$. For any $K\subset \mathcal{M}(X,T),$ denote $G^{K}=\{x\in X\mid K\subset V_T(x)\}$ and $G^{K}_{\varepsilon}=\{x\in X\mid \sup_{\mu\in K}\rho(\mu,V_T(x))\leq \varepsilon\},$ where $\rho(\mu,V_T(x))=\inf_{\nu\in V_T(x)}\rho(\mu,\nu).$
	For any nonempty compact connected subset $K\subset\cM(X,T)$, denote by $G_{K}=\{x\in X\mid V_T(x)=K\}$
	the saturated set of $K$. Note that for any $x\in X$, $V_T(x)$ is always a nonempty compact connected subset of $\cM(X,T)$ by \cite[Proposition 3.8]{DGS1976}, so $G_{K}\neq\emptyset$ requires that $K$ is a nonempty compact connected set. When $\sharp K>1,$ one has $G_K\subset G^K\subset IR(T),$ so we can obtain $IR(T)$ is residual in $X$ if $G_K$ is residual in $X.$ Therefore, saturated set is an important tool in the research of irregular set. In fact,  during the study of saturated sets, we find that  typical orbits and typical functions have extreme deviations. This will be seen in Theorem \ref{maintheorem-2}.
	\item If $X=M$ is a manifold, a probability measure $\mu \in \mathcal{M}(X)$ is called physical-like (or SRB-like or observable) if for any $\varepsilon>0$ the set
	$G^{\mu}_{\varepsilon}$ has positive Lebesgue measure. We denote by $\mathcal{O}_{T}$ the set of physical-like measures for $T.$ In \cite{CTV2019} the authors define the points with physical-like behaviour, that is $\{x\in X\mid V_T(x)\cap \mathcal{O}_{T}\neq\emptyset\}.$
\end{itemize}

Define
$\mathcal{M}_{dense}^*:=\{\mu\in\mathcal{M}(X,T)\mid  G^{\mu}_{\varepsilon} \text{ is dense in } X \text{ for any }\varepsilon>0\}.$ 
Now we state first result.
\begin{maintheorem}\label{maintheorem-1}
	Suppose that $(X,T)$ is a dynamical system with $\sharp \mathcal{M}_{dense}^*>1$. Then there exist a residual invariant subset $Y\subset X$ and an open and dense subset $\mathfrak{U}\subset C(X)$ such that for any $x\in Y$ and any $f\in \mathfrak{U}$ one has
	\begin{enumerate}
		\item $\mathcal{M}_{dense}^*\subset V_T(x);$
		\item $\underline{f}(x)\leq \inf\limits_{\mu\in \mathcal{M}_{dense}^*}\int f d\mu<\sup\limits_{\mu\in \mathcal{M}_{dense}^*}\int f d\mu\leq \overline{f}(x);$
		\item $V_T(x)\cap \mathcal{O}_{T}\neq\emptyset$ provided that $X=M$ is a manifold.
	\end{enumerate}
\end{maintheorem}

Remark that the most important observation in the proof is that for any $\mu\in  \mathcal{M}_{dense}^*,$ $G^{\mu}$ is residual in $X.$ This will be seen in the Section \ref{section-2}.  In addition, when $(X,T)$ is transitive, we will show that for any $\mu\in  \mathcal{M}_{dense}^*,$ there exists a transitive point $x$ such that $\mu\in V_T(x).$ Now, we need some definitions about transitive system.
A dynamical system is said to be transitive if $N(U,V):=\{n\in\mathbb{N}\mid U\cap{T^{-n}V}\neq\emptyset\}$ is nonempty for any nonempty open set $U$ and $V$. 
Denote the set of transitive points of $(X,T)$ by $$Trans(X,T):=\{x\in X\mid orb(x,T)\text{ is dense in X}\},$$ where $orb(x,T)=\{T^nx\mid n\in\mathbb{N}\}.$  It is known that $Trans(X,T)$ is either empty or residual in $X$ when there is no isolated point in $X$.
When $(X,T)$ is transitive, it's known that $Trans(X,T)=\{x\in X\mid \omega_T(x)\text{ is dense in X}\}$ is residual in $X.$
Next we define
$\mathcal{M}_{Trans}:=\bigcup\limits_{y\in Trans(X,T)}V_T(y).$ We denote $\mathcal{R}$
the set of real continuous functions whose irregular set is residual 
in $X$:
$$
\mathcal{R}:=\{f\in C(X)\mid I(f,T)\text{ is residual in } X\}.
$$
Now we state one result for transitive systems.
\begin{maintheorem}\label{maintheorem-2}
	Suppose that $(X,T)$ is a transitive dynamical system. Then $\mathcal{M}_{dense}^*=\mathcal{M}_{Trans}$  and there exists a residual invariant subset $Y\subset X$ such that for any $x\in Y$ and any $f\in C(X)$ one has
	\begin{enumerate}
		\item $x$ is extremely transitive that is $x\in Trans(X,T)$ and $V_T(x)=\mathcal{M}_{Trans};$
		\item $\underline{f}(x)=\inf\limits_{y\in Trans(X,T)}\underline{f}(y)= \inf\limits_{\mu\in \mathcal{M}_{Trans}}\int f d\mu\leq\sup\limits_{\mu\in \mathcal{M}_{Trans}}\int f d\mu=\sup\limits_{y\in Trans(X,T)}\underline{f}(y)= \overline{f}(x);$
		\item $V_T(x)\cap \mathcal{O}_{T}\neq\emptyset$ provided that $X=M$ is a manifold.
	\end{enumerate}
    In particular, if further $\sharp \mathcal{M}_{Trans}>1,$ then $\mathcal{R}$ is open and dense in $C(X),$ and $Y\subset\bigcap\limits_{f\in\mathcal{R}} I(f,T).$
\end{maintheorem}

We also obtain a result in subadditive sequences which can be applicative to cocycles.
A sequence $\{f_n\}_{n=1}^{\infty}$ of functions from $X$ to $\mathbb{R}$ will be called subadditive if, for each $x\in X$ and $m,n\geq 1,$ the inequality $f_{m+n}(x)\leq f_n(T^mx)+f_m(x)$ is satisfied. We denote $C(X,GL(d,\mathbb{R}))$ the space of continuous map $A:X\to GL(d,\mathbb{R})$ taking values into the space of $d\times d$ real invertible matrices with the norm $\|A\|:=\sup\limits_{x\in X}\|A(x)\|.$ Given $A\in C(X,GL(d,\mathbb{R})),$  let $A_n(x)=A(T^{n-1}x)\cdots A(Tx)A(x).$
The triple $(T,X,A)$ is a linear cocycle. For any $x\in X,$ denote $\overline{\chi}(x)=\limsup\limits_{n\to+\infty}\frac1n\log\|A_n(x)\|$ and $\underline{\chi}(x)=\liminf\limits_{n\to+\infty}\frac1n\log\|A_n(x)\|,$ if $\overline{\chi}(x)=\underline{\chi}(x),$ we denote it by $\chi(x).$
We denote $$LI_A:=\{x\in X\mid \lim\limits_{n\to+\infty}\frac1n\log\|A_n(x)\|\,\, \text{ diverges }\}$$ the set of Lyapunov-irregular points, and denote $\mathcal{C}_d,$  $\mathcal{R}_d$ the set of real continuous matrix functions, whose Lyapunov-irregular set is nonempty, residual in $X,$ respectively:
\[
\begin{split}
	\mathcal{C}_d&:=\{A\in C(X,GL(d,\mathbb{R}))\mid LI_A\neq\emptyset\},\\
	\mathcal{R}_d&:=\{A\in C(X,GL(d,\mathbb{R}))\mid LI_A\text{ is residual in } X\}.
\end{split}
\]
Obviously, $\mathcal{R}_d\subset\mathcal{C}_d$.
\begin{maintheorem}\label{maintheorem-3}
	Suppose that $(X,T)$ is a transitive dynamical system.
	\begin{enumerate}
		\item Given a subadditive sequence $\mathfrak{F}=\{f_n\}_{n=1}^{\infty}$ with $f_n\in C(X)$ for any $n\in\mathbb{N},$ if we denote $\mathfrak{F}_{\sup}^T=\sup\limits_{x\in Trans(X,T)}\limsup\limits_{n\to+\infty}\frac{1}{n}f_n(x),$ then $E_{\mathfrak{F}_{\sup}^T}=\{x\in Trans(X,T)\mid \limsup\limits_{n\to+\infty}\frac{1}{n}f_n(x)=\mathfrak{F}_{\sup}^T \}$ is residual in $X.$
		\item Given $A\in C(X,GL(d,\mathbb{R})),$
		if we denote $A_{\sup}^T=\sup\limits_{x\in Trans(X,T)}\overline{\chi}(x)$ and $A_{\inf}^T=\inf\limits_{x\in Trans(X,T)}\underline{\chi}(x),$ then $E_{A_{\sup}^T}=\{x\in Trans(X,T)\mid \overline{\chi}(x)=A_{\sup}^T \}$ and $\tilde{E}_{A_{\inf}^T}=\{x\in Trans(X,T)\mid \underline{\chi}(x)=A_{\inf}^T \}$ are residual in $X.$ Moreover, either $A_{\inf}^T=A_{\sup}^T,$ or $LI_A$ is residual in $X.$
		\item If further $\sharp \mathcal{M}_{Trans}>1,$ then there exist a residual invariant subset $Y\subset X$ and a dense subset $\mathfrak{U}\subset C(X,GL(d,\mathbb{R}))$ such that for any $x\in Y$ and any $A\in \mathfrak{U}$ one has
		\begin{enumerate}
			\item $x$ is extremely transitive that is $x\in Trans(X,T)$ and $V_T(x)=\mathcal{M}_{Trans};$
			\item $x\in LI_A;$
			\item $V_T(x)\cap \mathcal{O}_{T}\neq\emptyset$ provided that $X=M$ is a manifold.
		\end{enumerate}
	\end{enumerate}
\end{maintheorem}
\begin{remark}
	For any uniquely ergodic system, one has $\mathcal{R}=\emptyset.$ But by \cite{Furm1997} there exist a uniquely ergodic and minimal system $(X,T)$ such that $\mathcal{R}_d\neq\emptyset$ for any $d\geq 2.$ Recently, Carvalho and Varandas gave some criterion to show residual property of $I(f,T)$ in \cite{CarVar2021}. 
\end{remark}

In page 264 of his book \cite{Mane1987}, Ricardo Mañé wrote: “In general, (Lyapunov) regular points are very few from the
topological point of view – they form a set of first category.”
The third named author proved in \cite{Tian201702} that if $f: X \rightarrow X$ is a homeomorphism of a compact metric space $X$ with exponential specification property, $A: X \rightarrow \mathrm{GL}(d, \mathbb{R})$ is a Hölder continuous matrix function, then either $\int \chi(x)d\mu_1=\int \chi(x)d\mu_2$ for any $\mu_1,\mu_2\in \cM^{e}(X,T)$  or the Lyapunov-irregular set $LI_A$ is residual in $X$. According to Theorem \ref{maintheorem-3}, we have the following.
\begin{corollary}\label{coro-4}
	Suppose that $(X,T)$ is a transitive dynamical system. Given $A\in C(X,GL(d,\mathbb{R})),$ if there exist $\mu_1\neq\mu_2\in \cM^{e}(X,T)$ with $S_{\mu_1}=S_{\mu_2}=X$ and $\int \chi(x)d\mu_1<\int \chi(x)d\mu_2,$ then $A\in\mathcal{R}_d.$
\end{corollary}
\begin{corollary}\label{coro-6}
	Suppose that $(X,T)$ is a minimal dynamical system. Given $A\in C(X,GL(d,\mathbb{R})),$ then either there is $c\in \mathbb{R}$ such that $\chi(x)=c$ for any $x\in X,$ or $A\in\mathcal{R}_d.$
\end{corollary}
We recall that measure center is $C_T(X):=\overline{\cup_{\mu\in\mathcal{M}(X,T)}S_{\mu}},$ where $$S_\mu:=\{x\in X:\mu(U)>0\ \text{for any neighborhood}\ U\ \text{of}\ x\}$$ is the support of $\mu.$
\begin{remark}\label{rem-1}
	There are many systems with $\sharp \mathcal{M}_{dense}^*>1$.
	\begin{enumerate}
		\item If there exist $\mu_{1},\mu_{2}\in \cM^e(X,T)$ with $S_{\mu_{1}}=S_{\mu_{2}}$ and $\mu_{1}\neq\mu_{2},$ then $\mu_{1},\mu_{2}\in \mathcal{M}_{dense}^*$ and thus $\sharp \mathcal{M}_{dense}^*>1$. For examples,
		\begin{enumerate}
			\item Every minimal and non uniquely ergodic system. See examples of minimal systems with exactly two ergodic measures in \cite[Page 272]{DGS1976}.
			\item Every non uniquely ergodic system satisfying approximate product property and $C_T(X)=X.$ By \cite[Proposition 2.3, Theorem 2.1]{PS2005}, the ergodic measures of $(X,T)$ are entropy-dense, then $\cM^{e}(X,T)$ is dense in $\cM(X,T).$ From \cite[Proposition 5.7]{DGS1976}, $\cM^{e}(X,T)$ is a $G_\delta$-subset of $\cM(X,T).$ Then we have that $\cM^{e}(X,T)$ is residual in $\cM(X,T).$ From \cite[Lemma 5.1]{KOR2016}, there exists an invariant measure
			$\mu$ such that $S_\mu=C_T(X)$ for every dynamical system $(X,T).$ Then we can take $\nu\in \cM(X,T)$ such that $S_\nu=C_T(X)=X.$ Since $\{\mu\in\cM(X,T)\mid S_\mu=X\}$ is either empty or residual in $\cM(X,T)$ from \cite[Proposition 21.11]{DGS1976},  then $\{\mu\in\cM(X,T)\mid S_\mu=X\}$ is residual in $\cM(X,T)$ by $S_\nu=X.$ So $\{\mu\in\cM^{e}(X,T)\mid S_\mu=X\}$ is also residual in $\cM(X,T).$
		\end{enumerate}
	    \item For $(X,T),$ let $\Omega=\overline{\{T^nx_0\mid n\in\mathbb{N}\}}$ for some $x_0\in IR(T),$ then $\sharp \mathcal{M}_{dense}^*(T|_{\Omega})>1$ by $V_T(x_0)\subset \mathcal{M}_{dense}^*(T|_{\Omega})$ and $\sharp V_T(x_0)>1$. In \cite[Theorem 3]{Dowker1953}, Dowker proved that $IR(T)\cap\Omega$ is residual in $\Omega.$ By Theorem \ref{maintheorem-1}, we strengthen Dowker's resul  to that there exist a residual invariant subset $Y\subset \Omega$ and an open and dense subset $\mathfrak{U}\subset C(\Omega)$ such that for any $x\in Y$ and any $f\in \mathfrak{U}$ one has
	    \begin{enumerate}
	    	\item $V_T(x_0)\subset V_T(x).$
	    	\item $\underline{f}(x)\leq \underline{f}(x_0)<\overline{f}(x_0)\leq \overline{f}(x).$
	    \end{enumerate}
    \item Every non-trival homoclinic class. Recall that for a hyperbolic periodic orbit $\mathcal{O}(p),$  the homoclinic class of $\mathcal{O}(p)$ is the set
    $H(p):=\overline{W^{s}(\mathcal{O}(p)) \pitchfork W^{u}(\mathcal{O}(p))}$ $($for example, see \cite[Definition 2.1]{ABC2011} for more detailed definition$).$ Denote $\mu_{p}$ the periodic measure supported on $\mathcal{O}(p),$ then it's easy to see that $H(p)=\overline{W^{s}(\mathcal{O}(p)) \cap H(p)}$ and $W^{s}(\mathcal{O}(p)) \cap H(p)\subset G_{\mu_{p}} \cap H(p).$ This means that $G_{\mu_{p}} \cap H(p)$ is dense in $H(p)$ and thus $\mu_{p}\in \mathcal{M}_{dense}^*.$ Since $H(p)$ is non-trival, there is a hyperbolic periodic orbit $\mathcal{O}(q)$ such that $p$ and $q$ are homoclinically related and $\mu_{p}\neq\mu_{q}.$ Then one has $\mu_{q}\in \mathcal{M}_{dense}^*$ by $H(p)=H(q).$ So $\sharp \mathcal{M}_{dense}^*>1$.
    \item All the above results can be applied to flow if one replaces Birkhoff averages by the suitable means obtained by integration along the orbits of the flow. We consider the well-known example of a planar flow with divergent time averages attributed to Bowen $($see \cite{Takens1995}$)$. In the example, there are two fixed
    hyperbolic saddle points A, B, and heteroclinic orbits connecting the two saddle points. We denote the expanding and contracting eigenvalues of the linearized vector field in $A$ by $\alpha_{+}$ and $-\alpha_{-}$ and in $B$ by $\beta_{+}$ and $-\beta_{-}$. The condition on the eigenvalues which makes the cycle attracting is that the contracting eigenvalues dominate: $\alpha_{-} \beta_{-}>$ $\alpha_{+} \beta_{+}.$
    Denote
    $$
    \lambda=\alpha_{-} / \beta_{+} \text {and } \sigma=\beta_{-} / \alpha_{+},
    $$
    then their values are positive and their product is bigger than 1.
    For any point $x$ in the interior of the plane curve formed by the heteroclinic orbits connecting the fixed
    hyperbolic saddle points A, B, $V_T(x)=\{t\mu_{1}+(1-t)\mu_{2}\mid t\in[0,1]\}$ where $\mu_1=\frac{\sigma}{1+\sigma} \delta_A+\frac{1}{1+\sigma} \delta_B$ and $\mu_{2}=\frac{\lambda}{1+\lambda} \delta_B+\frac{1}{1+\lambda} \delta_A.$ Since $\lambda\sigma>1,$ then $\mu_{1}\neq\mu_{2}.$ By the arbitrariness of $x,$ we have $\{t\mu_{1}+(1-t)\mu_{2}\mid t\in[0,1]\}\subset \mathcal{M}_{dense}^*$ and thus $\sharp \mathcal{M}_{dense}^*>1$.
	\end{enumerate}
\end{remark}

Next, we will introduce a new abstract property in base space of dynamical systems which give a sufficient condition for that $\sharp \mathcal{M}_{dense}^*>1$.
Before that, we recall specification-like properties.
The specification property ($SP$) introduced by Bowen \cite{Bowen1971} plays an important role in hyperbolic dynamics and thermodynamical formalism. Based on this property, a number of classical results have been established, including growth rate of periodic orbits \cite{Bowen1971}, uniqueness of equilibrium states \cite{Bowen1974} and so on.
In recent years this notion served as a basis for many developments in the theory of dynamical systems, especially in multifractal analysis \cite{Thompson2010,Tian2017,CT2021,DTY2015,LW2014}. For systems with specification property, Li and Wu showed in \cite{LW2014} that $I(f,T)$ is either empty or residual for any $f\in C(X)$. Dong, Tian and Yuan proved in \cite{DTY2015} that for systems with asymptotic average shadowing property and measure center coinciding with $X$, $I(f,T)$ is either empty or residual for any $f\in C(X)$ and $\bigcap\limits_{f\in\mathcal{C}} I(f,T)$ is residual in $X$. 
Weak variations of specification property, called specification-like properties, have been introduced to study broader classes of systems, such as the almost specification property ($ASP$)\cite{Date1990,Marcus1980}, the approximate product property ($ApPP$)\cite{PS2005} and the almost product property ($AlPP$)\cite{PS2007,KKO2014}. It is known that (see \cite{Yama2009} for more information)
$$SP\Rightarrow ASP\Rightarrow ApPP\ and\ SP\Rightarrow AlPP\Rightarrow ApPP.$$
We say that $(X,T)$ has the approximate product property ($ApPP$), if for any $\varepsilon>0$, $\delta_1>0$ and $\delta_2>0$, there exists $N(\varepsilon,\delta_1,\delta_2)\in\mathbb{N^{+}}$ such that for any $n\geq N$ and any sequence $\{x_i\}_{i=1}^{+\infty}$ of $X$, there exist a sequence of integers $\{h_i\}_{i=1}^{+\infty}$ and $x\in
X$ satisfying $h_1=0$, $n\leq h_{i+1}-h_i\leq n(1+\delta_2)$ and $$\#\{0\leq j\leq n-1\mid d(T^{h_i+j}x,T^jx_i)\geq\varepsilon\}\leq\delta_1n.$$

Now, we will introduce a new property of dynamical systems. It is weaker than most known versions of specification-like properties.
For any $\varepsilon>0$, any $x,y\in X$ and any $n\in\mathbb{N^+}$, we can define two integer value functions $m=m(\varepsilon,x,y,n)\in\mathbb{N}$, $g=g(\varepsilon,x,y,n)\in\mathbb{N}$ such that $g(\varepsilon,x,y,n)\leq n$. Denote $\Lambda_n:=\{0,\cdots,n-1\}$.

\begin{definition}
	Suppose that $(X,T)$ is a dynamical system. We say $(X,T)$ has $m$-$g$-product property, if for any $\varepsilon>0$, $x_1,x_2\in X$, there exists $N(\varepsilon,x_1,x_2)\in\mathbb{N^{+}}$ and for any positive integer $n\geq N$, there exist $z\in X,$ $p\in\mathbb{N},$ $\Lambda\subseteq \Lambda_n$, such that $0\leq p\leq m(\varepsilon,x_1,x_2,n),$ $|\Lambda_n\setminus \Lambda|\leq g(\varepsilon,x_1,x_2,n),$ and
	$$d(z,x_1)<\varepsilon,\ d(T^{p+j}z,T^jx_2)<\varepsilon\ \text{for any}\ j\in \Lambda.$$
	If $(X,T)$ has $m$-$g$-product property and is transitive, we say $(X,T)$ is $m$-$g$-transitive. If $m\equiv0$ or $g\equiv0$, we will omit $m$ or $g$.
\end{definition}
For any transitive system, we can find an integer function $m(\varepsilon,x,y,n)$ such that it is a $m$-transitive system. However, we can't obtain $\sharp \mathcal{M}_{dense}^*>1$ if $(X,T)$ is only transitive and not uniquely ergodic. 	In \cite[Section 4]{KW1981}, Katznelson and Weiss constructed a transitive dynamical system $(X,T)$ such that
\begin{enumerate}
	\item $(X,T)$ is not uniquely ergodic.
	\item every point in $X$ is a generic point of an ergodic measure.
	\item there exists a uniquely ergodic measure $\mu$ with $S_\mu=X.$
	\item For any $\nu\in \mathcal{M}^e(X,T)\setminus \{\mu\},$ $G_{\nu}\cap Trans(X,T)=\emptyset.$
\end{enumerate}
Then $\mathcal{M}_{dense}^*=\mathcal{M}_{Trans}=\{\mu\}$ and thus $\sharp \mathcal{M}_{dense}^*=1.$ So we need to put some restrictions on $m$-$g$-product property if we want to obtain $\sharp \mathcal{M}_{dense}^*>1.$

\begin{maintheorem}\label{maintheorem-4}
	Suppose that $(X,T)$ is a dynamical system.
	\begin{enumerate}
		\item Let $0<\kappa<\frac{1}{2},$ if $(X,T)$ is not uniquely ergodic and is $m$-$g$-transitive with $$\liminf\limits_{n\to+\infty}\frac{m(\varepsilon,x,y,n)+g(\varepsilon,x,y,n)}{m(\varepsilon,x,y,n)+n}<\kappa< \frac{1}{2}$$ for any $\varepsilon>0,$ any $x,y\in X.$ Then one has $\sharp \mathcal{M}_{dense}^*>1;$
		\item If $(X,T)$ is not uniquely ergodic and satisfies $m$-$g$-product property with $\liminf\limits_{n\to+\infty}\frac{m(\varepsilon,x,y,n)+g(\varepsilon,x,y,n)}{n}=0$ for any $\varepsilon>0,$ any $x,y\in X$, then $\mathcal{M}^{e}(X,T)\subset\mathcal{M}_{dense}^*$ and thus $\sharp \mathcal{M}_{dense}^*>1;$
		\item If $(X,T)$ satisfies approximate product property and $C_T(X)=X$, then $\mathcal{M}_{dense}^*=\mathcal{M}(X,T)$ and $(X,T)$ is $m$-$g$-transitive with $\lim\limits_{n\to+\infty}\frac{m(\varepsilon,x,y,n)+g(\varepsilon,x,y,n)}{n}=0$ for any $\varepsilon>0,$ any $x,y\in X;$
		\item If $(X,T)$ is minimal, then $\mathcal{M}^{e}(X,T)\subset\mathcal{M}_{dense}^*$ and for any $\varepsilon>0$, there is $m(\epsilon)\in\mathbb{N}$ such that for any $x_1,x_2\in X$, any $n\in\mathbb{N^{+}},$ and any $\varepsilon'>0$ there is $l\in\mathbb{N}$ and $z\in X$ such that $0\leq l\leq m(\varepsilon)$, $$d(z,x_1)<\varepsilon,\ d(T^{p+j}z,T^jx_2)<\varepsilon'\ \text{for any}\ j\in \Lambda_n.$$ In particular, $(X,T)$ is $m$-transitive with $\lim\limits_{n\to+\infty}\frac{m(\varepsilon)}{n}=0$. 
	\end{enumerate}
\end{maintheorem}

Next, we give an example which satisfies Theorem \ref{maintheorem-4}(1) but does not satisfy Theorem \ref{maintheorem-4}(2)-(4).
\begin{maintheorem}\label{maintheorem-5}
	For any $\kappa>0$, there is a $m$-transitive dynamical system $(X,T)$ such that
	\begin{enumerate}
		\item $\lim\limits_{n\to+\infty}\frac{m(\varepsilon,x,y,n)}{n}=\kappa$, and for any $0<\tilde{\kappa}<\kappa$, $\tilde{m}(\varepsilon,x,y,n)$ with $\lim\limits_{n\to+\infty}\frac{\tilde{m}(\varepsilon,x,y,n)}{n}=\tilde{\kappa}$, $(X,T)$ is not $\tilde{m}$-transitive;
		\item The set of periodic points of $(X,T)$ is dense in $X$ and thus $C_T(X)=X$;
		\item $(X,T)$ has positive topological entropy;
		\item $(X,T)$ does not satisfy the approximate product property;
		\item $(X,T)$ is not minimal and is not uniquely ergodic.
	\end{enumerate}
\end{maintheorem}

\begin{figure}[h]\caption{Relationships between transitivity,  approximate product property, minimal property and $m$-$g$-product property.}\label{fig-1}
	\begin{center}
		\begin{tikzpicture}
			\draw (0,0) rectangle (7,5);
			\draw (3.5,2.4) ellipse (3.3cm and 2.1cm);
			\draw (2.0,2.9) ellipse (1.3cm and 0.8cm);
			\draw (4.9,2.9) ellipse (1.3cm and 0.8cm);
			\node at(3.5,5.3){$DS_{trans}$};
			\node at(1.9,2.9){$DS_{m}$};
			\node at(4.9,2.9){$DS_{app}$};
			\node at(3.5,1.2){$DS_{mg}$};
		\end{tikzpicture}
	\end{center}
\end{figure}
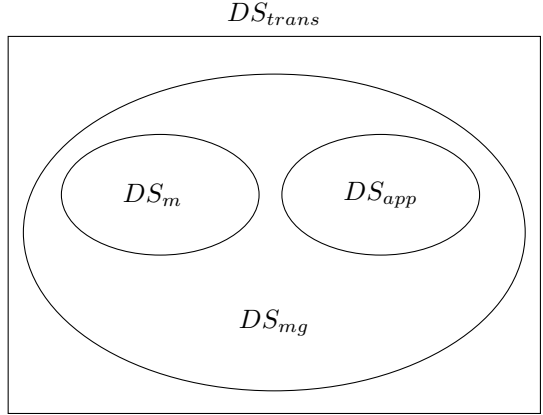
Finally, we describe the relationships between transitivity, approximate product property, minimal property and $m$-$g$-product property as Figure \ref{fig-1}. Let
$DS_{trans}$ denote the set of transitive systems;  $DS_{mg}$ denote the set of systems which are $m$-$g$-transitive with $\lim\limits_{n\to+\infty}\frac{m(\varepsilon,x,y,n)+g(\varepsilon,x,y,n)}{n}=0,$ for any $\varepsilon>0,$ any $x,y\in X$; $DS_{app}$ denote the set of systems which satisfy $C_T(X)=X$ and approximate product property; $DS_{m}$ denote the set of minimal systems.

\textbf{Organization of this paper.} In Section \ref{section-2} we study properties of saturated-like sets, relationship between irregular points and saturated-like sets, and give the proof of Theorem \ref{maintheorem-1}. In Section \ref{section-3} we show that $\mathcal{M}_{dense}^*=\mathcal{M}_{Trans}$ for transitive systems and give the proof of Theorem \ref{maintheorem-2}. In section \ref{section-4}, we consider irregular points of subadditive sequences and cocycles, and we give the proof of Theorem \ref{maintheorem-3}.
In section \ref{section-5}, we discuss relationship between $m$-$g$-product property and approximate product property and minimality, and we give the proof of Theorem \ref{maintheorem-4}.
In section \ref{section-6}, we give an example which satisfies $m$-$g$-product property but does not satisfy approximate product property and minimality.

\section{Saturated-like sets and proof of Theorem \ref{maintheorem-1}}\label{section-2}
\subsection{Saturated-like sets}
For any $K\subset \mathcal{M}(X,T),$ denote $^{K}G=\{x\in X\mid V_T(x)\cap K\neq\emptyset\}.$
Denote a ball in $\mathcal{M}(X)$ by $\mathcal{B}(\nu, \zeta)=\{\mu \in \mathcal{M}(X)\mid \rho(\nu, \mu) < \zeta\}.$ For any $K\subset \mathcal{M}(X,T),$  denote $\mathcal{B}(K, \zeta)=\{\mu \in \mathcal{M}(X)\mid \rho(\mu, K) < \zeta\}.$ 
\begin{theorem}\label{thm-1}
	Suppose that $(X,T)$ is a dynamical system. Then for any $K\subset \mathcal{M}(X,T),$ 
	\begin{enumerate}
		\item if $K\neq \emptyset$ and $\{x\in X\mid V_T(x)\cap \mathcal{B}(K, \varepsilon)\neq\emptyset \}$ is dense in $X$ for any $\varepsilon>0,$ then $^{\overline{K}}G$ is residual in $X;$
		\item if $\{x\in X\mid V_T(x)\cap \mathcal{B}(\mu, \varepsilon)\neq\emptyset \}$ is dense in $X$ for any $\varepsilon>0$ and any $\mu\in K,$ then $G^{\overline{K}}$ is residual in $X.$
	\end{enumerate}
\end{theorem}
\begin{proof}
	For any nonempty open set $V\subset\mathcal{M}(X,T)$ and $N_0\in\mathbb{N}$, let $$U(N_0,V):=\bigcup_{N>N_{0}}\left\{x \in X \mid \frac{1}{N} \sum_{j=0}^{N-1} \delta_{T^{j} x} \in V\right\},$$ it's easy to obtain the following lemma.
	\begin{lemma}\label{lemma-13}
		$U(N_0,V)$ is an open set.
	\end{lemma}
	(1) If $\{x\in X\mid V_T(x)\cap V\neq\emptyset\}$ is dense in $X,$ then $U(N_0,V)$ is an open and dense set in $X$ by $\{x\in X\mid V_T(x)\cap V\neq\emptyset\}\subset U(N_0,V).$ Then for any $m\in\mathbb{N^+}$ and $N_0\in\mathbb{N},$ $U(N_0,\mathcal{B}(K, \frac{1}{m}))$ is an open and dense set in $X.$ Note that $\bigcap_{N_0=0}^{+\infty}U(N_0,\mathcal{B}(K, \frac{1}{m}))\subset \{x\in X\mid V_T(x)\cap \overline{\mathcal{B}(K, \frac{1}{m})}\neq\emptyset \},$ then  $^{\overline{\mathcal{B}(K, \frac{1}{m})}}G$ is residual in $X.$ Since $\bigcap_{m=0}^{+\infty} {~}^{\overline{\mathcal{B}(K, \frac{1}{m})}}G\subset {~}^{\overline{K}}G,$ we have $^{\overline{K}}G$ is residual in $X.$
	
	(2) If $K=\emptyset,$ then $G^{\overline{K}}=X.$ Otherwise, $\overline{K}\subset \cM(X,T)$ be a nonempty compact set. We can find open balls $\{V_i\}_{i\in\mathbb{N^+}}$ and $\{U_i\}_{i\in\mathbb{N^+}}$ in $\cM(X)$ such that
	\begin{enumerate}[(a)]
		\item $V_i\subset\overline{V}_i\subset U_i$;
		\item $diam(U_i)\longrightarrow0$;
		\item $V_i\cap \overline{K}\neq\emptyset$;
		\item each point of $\overline{K}$ lies in infinitely many $V_i$.
	\end{enumerate}
	Put $P(U_i):=\{x\in X\mid V_T(x)\cap U_i\neq\emptyset\}$, then $\bigcap\limits_{i=1}^{+\infty}P(U_i)= G^{\overline{K}},$  and $$
	P\left(U_{i}\right) \supset \bigcap_{N_{0}=1}^{+\infty} \bigcup_{N>N_{0}}\left\{x \in X \mid \frac{1}{N} \sum_{j=0}^{N-1} \delta_{T^{j} x} \in V_{i}\right\}
	.$$ Since $U(N_0,V_i)$ is an open set by Lemma \ref{lemma-13}, by Baire category theorem if we have
	$$\forall N_0\in\mathbb{N^+},\ \forall i\in\mathbb{N^+},\ U(N_0,V_i)\text{ is dense in }X
	,$$
	then $G^{\overline{K}}$ is residual in $X.$  
	
	Since $V_i$ is open, there exists $\nu\in\cM(X,T)\cap V_i$ such that $\nu\in K.$ Then $\{x\in X\mid V_T(x)\cap \mathcal{B}(\nu, \varepsilon)\neq\emptyset \}$ is dense in $X$ for any $\varepsilon>0.$ It's easy to see that $\{x\in X\mid V_T(x)\cap \mathcal{B}(\nu, \varepsilon)\neq\emptyset \}\subset U(N_0,V_i)$ for small enough $\varepsilon>0.$ This means $U(N_0,V_i)$ is dense in $X.$ So $G^{\overline{K}}$ is residual in $X.$
\end{proof}

According to Theorem \ref{thm-1}, we have some corollaries.
\begin{corollary}\label{coro-2}
	Suppose that $(X,T)$ is a dynamical system. If $X=M$ is a manifold, then $^{\mathcal{O}_{T}}G$ is residual in $X.$
\end{corollary}
\begin{proof}
	By \cite[Theorem 1.3]{CE2011} and \cite[Theorem 1.5]{CE2011}, $\mathcal{O}_{T}$ is a nonempty compact set and $\{x\in X\mid V_T(x)\subset \mathcal{O}_{T}\}$ has total Lebesgue measure. Then $\{x\in X\mid V_T(x)\cap \mathcal{B}(\mathcal{O}_{T}, \varepsilon)\neq\emptyset \}$ is dense in $X$ for any $\varepsilon>0.$ So by Theorem \ref{thm-1}(1), $^{\mathcal{O}_{T}}G$ is residual in $X.$
\end{proof}

\begin{lemma} \label{lemma-1}
	Suppose that $(X,T)$ is a dynamical system. Then $G^{K}=G^{\overline{K}}$ for any $K\subset \mathcal{M}(X,T)$.
\end{lemma}
\begin{proof}
	On one hand, by $K\subset \overline{K},$ we have $G^{\overline{K}}\subset G^{K}.$ On the other hand, since $V_T(x)$ is a nonempty compact connected set for any $x\in X$ by \cite[Proposition 3.8]{DGS1976}, then $V_T(y)\supset \overline{K}$ for any $y\in G^{K}.$ This means $G^{K}\subset G^{\overline{K}}.$
\end{proof}

Recall that $\mathcal{M}_{dense}^*:=\{\mu\in\mathcal{M}(X,T)\mid  G^{\mu}_{\varepsilon} \text{ is dense in } X \text{ for any }\varepsilon>0\},$ then by Theorem \ref{thm-1}(2) and Lemma \ref{lemma-1} one has the following.
\begin{corollary}\label{coro-1}
	Suppose that $(X,T)$ is a dynamical system, then $G^{\mathcal{M}_{dense}^*}$ is residual in $X.$
\end{corollary}

For any $\mu\in \mathcal{M}(X,T),$ by Theorem \ref{thm-1}(2) one has the following.
\begin{corollary}\label{coro-3}
	Suppose that $(X,T)$ is a dynamical system. For any $\mu\in \mathcal{M}(X,T),$ if $\{x\in X\mid \rho(\mu, V_T(x))< \varepsilon \}$ is dense in $X$ for any $\varepsilon>0,$ then $\{x\in X\mid \rho(\mu, V_T(x))=0 \}$ is residual in $X.$
\end{corollary}

\subsection{Saturated-like sets and irregular points}
For any $K\subset \mathcal{M}(X,T),$ we denote $\mathcal{C}_{K}:=\{f\in C(X)\mid \inf\limits_{\mu\in K}\int f\mathrm{d\mu}<\sup\limits_{\mu\in K}\int f\mathrm{d\mu}\}.$ Then we have the following theorem.
\begin{theorem}\label{thm-2}
	Suppose that $(X,T)$ is a dynamical system. Then for any $K\subset \mathcal{M}(X,T)$ with $\sharp K>1,$ $\mathcal{C}_{K}$ is  a nonempty open and dense subset in $C(X),$ and $G^{K}\subset \bigcap\limits_{f\in\mathcal{C}_{K}} I(f,T)\subset IR(T)\subsetneq X.$ If further $G^{K}$ is dense in $X,$ then $\mathcal{C}_{K}\subset \mathcal{R}.$ If further $G_{K}$ is residual in $X,$ then $\mathcal{C}_{K}= \mathcal{R}.$
\end{theorem}
\begin{proof}
	Since $\sharp K>1,$ then there are $\mu\neq\nu\in K$ and $f\in C(X)$ such that $\int f\mathrm{d\mu}\neq\int f \mathrm{d\nu}.$ Then $f\in \mathcal{C}_{K}$ and thus $\mathcal{C}_{K}$ is nonempty.  We fix $f_0\in \mathcal{C}_{K}$ and $\mu_{1},\mu_{2}\in K$ with $\int f_0\mathrm{d\mu_1}< \int f_0\mathrm{d\mu_2}.$
	On one hand, we show $\mathcal{C}_{K}$  is  dense in $C(X)\setminus \mathcal{C}_{K}.$ Fix $f\in C(X)\setminus \mathcal{C}_{K}.$ Take $f_n=\frac 1n f_0 + f,\,n\geq 1.$ Then $f_n$ converges to $f$ in sup norm. By construction, it is easy to check that $f_n\in \mathcal{C}_{K}.$ Thus $\mathcal{C}_{K}$  is  dense in $C(X).$
	
	On the other hand, we prove that $\mathcal{C}_{K}$ is open. Fix $f\in \mathcal{C}_{K}.$  Then   there must exist two different invariant measures $\mu_1,\mu_2\in K$ such that $ \int f d\mu_1< \int f d\mu_2.$  By continuity of sup norm, we can take an open neighborhood of $f$, denoted by $U (f)$, such that for any $g\in U (f),$  $\int g d\mu_1< \int g d\mu_2.$  Then $U (f)\subset \mathcal{C}_{K}.$ So $\mathcal{C}_{K}$ is open in $C(X).$
	
	For any $x\in G^{K},$ one has $K\subset V_T(x).$ Then $\liminf\limits_{n\to+\infty}\frac{1}{n}\sum\limits_{i=0}^{n-1}f(T^ix)\leq \inf\limits_{\mu\in K}\int f\mathrm{d\mu}<\sup\limits_{\mu\in K}\int f\mathrm{d\mu}\leq \limsup\limits_{n\to+\infty}\frac{1}{n}\sum\limits_{i=0}^{n-1}f(T^ix)$ for any $f\in \mathcal{C}_{K}.$ Thus $x\in I(f,T).$ So we have $G^{K}\subset \bigcap\limits_{f\in\mathcal{C}_{K}} I(f,T)\subset IR(T)\subsetneq X.$
	
	If further $G^{K}$ is dense in $X,$ by Theorem \ref{thm-1}(2) and Lemma \ref{lemma-1} one has  $G^{K}$ is residual in $X.$ Then $\bigcap\limits_{f\in\mathcal{C}_{K}} I(f,T)$ is residual in $X$ which implies $\mathcal{C}_{K}\subset \mathcal{R}.$
	
	If further $G_{K}$ is residual in $X,$ we show $\mathcal{R}\subset\mathcal{C}_{K}.$ For any $f\in \mathcal{R},$ $I(f,T)\cap G_K$ is residual in $X.$ Take $x\in I(f,T)\cap G_K,$ then there exist $\mu_1,\mu_{2}\in V_T(x)=K$ such that $\int fd\mu_1\neq \int fd\mu_2.$ So $f\in \mathcal{C}_{K}.$
\end{proof}

\subsection{Proof of Theorem \ref{maintheorem-1}} Let $Y=G^{\mathcal{M}_{dense}^*}$ and $\mathfrak{U}=\mathcal{C}_{\mathcal{M}_{dense}^*}.$ Then by Corollary \ref{coro-1} and Theorem \ref{thm-2}, $Y$ is residual in $X$ and $\mathfrak{U}$ is open and dense in $C(X).$ By Theorem \ref{thm-2},    for any $x\in Y$ and any $f\in \mathfrak{U}$ one has $\mathcal{M}_{dense}^*\subset V_T(x)$ and $\underline{f}(x)\leq \inf\limits_{\mu\in \mathcal{M}_{dense}^*}\int f d\mu<\sup\limits_{\mu\in \mathcal{M}_{dense}^*}\int f d\mu\leq \overline{f}(x).$ Then we have item (1) and (2).

If $X=M$ is a manifold, let $Y={~}^{\mathcal{O}_{T}}G\cap G^{\mathcal{M}_{dense}^*}.$ By Corollary \ref{coro-2} and Corollary \ref{coro-1}, $Y$ is residual in $X$ and $V_T(x)\cap \mathcal{O}_{T}\neq\emptyset$ for any $x\in Y.$  Then we have item (3).  \qed

\subsection{Some facts} In this subsection, we discuss some independent facts.
\begin{proposition}
	Suppose that $(X,T)$ is a dynamical system. Then for any compact subset $K\subset \mathcal{M}(X,T),$ $G^{K},$ $^{K}G$ and $G_K$ are all Borel sets.
\end{proposition}
\begin{proof}
	If $K=\emptyset,$ $G^{K}=X,$ $^{K}G=\emptyset,$ $G_K=\emptyset$ are all Borel sets. Thus we assume that $K\neq\emptyset.$
	Since $K$ is compact, there exist open balls $U_i$ in $\cM(X)$ such that
	\begin{description}
		\item[(a)] $\lim\limits_{i\to\infty}\mathrm{diam}(U_i)=0;$
		\item[(b)] $U_i\cap K\neq \emptyset\ \forall i\in\mathbb{N^{+}};$
		\item[(c)] each point of $K$ lies in infinitely many $U_i$.
	\end{description}
	Put $P(U_{i})=\left\{x \in X : V_T(x) \cap U_i \neq \emptyset\right\}.$
	It is easy to see that $G^{K}=\bigcap_{i=1}^{\infty} P(U_i)$ and $^{K}G=\bigcap_{i=1}^{\infty}\bigcup_{j=i}^{\infty} P(U_j).$ Now for any $i\in\mathbb{N^{+}},$ we assume that $U_i=\mathcal{B}(\nu_i,\zeta_i)$ for some $\nu_i\in \cM(X)$ and some $\zeta_i>0.$ Let $U_i^l=\mathcal{B}(\nu_i,(1-\frac{1}{l})\zeta_i)$ for any $l\in\mathbb{N^{+}},$ then we have 
	\begin{equation*}
		P(U_{i})=\bigcup_{l=1}^{\infty}\bigcap_{N=0}^{\infty}\bigcup_{n=N}^{\infty}\{x\in X\mid\frac{1}{n}\sum_{j=0}^{n-1}\delta_{T^{j}x}\in U_i^l\}.
	\end{equation*}
	Since $x\to \frac{1}{n}\sum_{j=0}^{n-1}\delta_{T^{j}x}$ is continuous for fixed $n\in\mathbb{N},$ the sets $\{x\in X\mid\frac{1}{n}\sum_{j=0}^{n-1}\delta_{T^{j}x}\in U_i^l\}$ are open. So $P(U_{i})$ is a Borel set for any $i\in\mathbb{N^{+}}$ which implies $G^{K}$ and $^{K}G$ are Borel sets. 
	
	On the other hand, it is easy to see that $\{x\in X\mid V_T(x)\subset K\}=\bigcap_{l=1}^{\infty}\bigcup_{N=0}^{\infty}\bigcap_{n=N}^{\infty}\{x\in X\mid\frac{1}{n}\sum_{j=0}^{n-1}\delta_{T^{j}x}\in \mathcal{B}(K,\frac{1}{l})\}.$ So $\{x\in X\mid V_T(x)\subset K\}$ is a Borel set which implies $G_K=G^K\cap \{x\in X\mid V_T(x)\subset K\}$ is a Borel set. 
\end{proof}

\begin{proposition}\label{prop-2}
	Suppose that $(X,T)$ is a dynamical system. Then $\mathcal{M}_{dense}^*$ is a compact subset of $\mathcal{M}(X,T)$. If $(X,T)$ is transitive, then $\mathcal{M}_{Trans}$ is a nonempty compact connected subset of $\mathcal{M}(X,T).$
\end{proposition}
\begin{proof}
	First, we show that $\overline{\mathcal{M}_{dense}^*}\subset \mathcal{M}_{dense}^*$ which implies $\mathcal{M}_{dense}^*$ is compact. For any $\mu\in \overline{\mathcal{M}_{dense}^*}$ and $\varepsilon>0,$ there exists $\nu\in \mathcal{M}_{dense}^*$ such that $\rho(\mu,\nu)<\frac{\varepsilon}{2}.$ Then $G^{\nu}_{\frac{\varepsilon}{2}}$ is dense in $X.$ For any $x\in G^{\nu}_{\frac{\varepsilon}{2}},$ one has $\rho(\nu, V_T(x))\leq \frac{\varepsilon}{2}.$ Then $\rho (\mu,V_T(x))<\rho(\mu,\nu)+\rho(\nu,V_T(x))=\varepsilon.$ So $G^{\nu}_{\frac{\varepsilon}{2}}\subset G^{\mu}_{\varepsilon}.$ This means $\mu\in \mathcal{M}_{dense}^*.$
	
	Next, we show that $\mathcal{M}_{Trans}$ is a nonempty compact connected subset of $\mathcal{M}(X,T).$ If $(X,T)$ is transitive, $\mathcal{M}_{Trans}$ is nonempty. Then $\mathcal{M}_{dense}^*$ is nonempty by $\mathcal{M}_{dense}^*\supset \mathcal{M}_{Trans}.$
	By Corollary \ref{coro-1}, $G^{\mathcal{M}_{dense}^*}$ is residual in $X.$ Then $G^{\mathcal{M}_{dense}^*}\cap Trans(X,T)$ is residual in $X.$ Since $\mathcal{M}_{dense}^*\supset \mathcal{M}_{Trans},$ we have $G^{\mathcal{M}_{dense}^*}\subset G^{\mathcal{M}_{Trans}}.$ Then $G^{\mathcal{M}_{Trans}}\cap Trans(X,T)$ is also residual in $X.$ Then there exists $x\in Trans(X,T)$ such that $\mathcal{M}_{Trans}\subset V_T(x)$ which implies $\mathcal{M}_{Trans}= V_T(x).$ So $\mathcal{M}_{Trans}$ is compact and connected since $V_T(x)$ is a nonempty compact connected set by \cite[Proposition 3.8]{DGS1976}. 
\end{proof}

\section{Transitive systems and proof of Theorem \ref{maintheorem-2}}\label{section-3}

\subsection{$\mathcal{M}_{dense}^*=\mathcal{M}_{Trans}$} 
In this subsection, we prove that $\mathcal{M}_{dense}^*=\mathcal{M}_{Trans}$ for transitive systems.
\begin{proposition}\label{prop-3}
	Suppose that $(X,T)$ is a transitive dynamical system. Then $\mathcal{M}_{dense}^*=\mathcal{M}_{Trans}\neq\emptyset.$
\end{proposition}
\begin{proof}
	It's easy to see that $\mathcal{M}_{dense}^*\supset \mathcal{M}_{Trans}.$ Next, we show that $\mathcal{M}_{dense}^*\subset \mathcal{M}_{Trans}.$
	Since $(X,T)$ is transitive, $\mathcal{M}_{Trans}$ is nonempty. For any $\mu \in \mathcal{M}_{dense}^*,$ $G^{\mu}_{\varepsilon}$ is dense in $X$ for any $\varepsilon>0.$ Then by Corollary \ref{coro-3}, $G^{\mu}$ is residual in $X.$ Thus $Trans(X,T)\cap G^{\mu}$ is also residual in $X.$ So $\mu\in \mathcal{M}_{Trans}.$ This means $\mathcal{M}_{dense}^*=\mathcal{M}_{Trans}.$
\end{proof}

\subsection{Proof of Theorem \ref{maintheorem-2}}
By Proposition \ref{prop-3}, one has $\mathcal{M}_{dense}^*=\mathcal{M}_{Trans}.$   Let $Y=G^{\mathcal{M}_{dense}^*}\cap Trans(X,T).$ Then $Y$ is residual in $X$ by  Corollary \ref{coro-1}.
For any $x\in Y,$ one has $\mathcal{M}_{Trans}= V_T(x)$ by $\mathcal{M}_{Trans}=\mathcal{M}_{dense}^*\subset V_T(x)$ and $x\in Trans(X,T).$  Then we have item (1).
Since $x\in Trans(X,T)$ and  $\mathcal{M}_{Trans}= V_T(x)$, we have $\inf\limits_{y\in Trans(X,T)}\underline{f}(y)\leq \underline{f}(x)\leq \inf\limits_{\mu\in \mathcal{M}_{Trans}}\int f d\mu\leq\sup\limits_{\mu\in \mathcal{M}_{Trans}}\int f d\mu\leq \overline{f}(x)\leq\sup\limits_{y\in Trans(X,T)}\overline{f}(y)$. For any $y\in Trans(X,T),$ one has $V_T(y)\subset \mathcal{M}_{Trans}$ which implies $\inf\limits_{\mu\in \mathcal{M}_{Trans}}\int f d\mu\leq \inf\limits_{\mu\in V_T(y)}\int f d\mu\leq \underline{f}(y)$ and $\overline{f}(y)\leq \sup\limits_{\mu\in V_T(y)}\int f d\mu\leq \sup\limits_{\mu\in \mathcal{M}_{Trans}}\int f d\mu.$ Then we have item (2). If $X=M$ is a manifold, let $Y={~}^{\mathcal{O}_{T}}G\cap G^{\mathcal{M}_{dense}^*}\cap Trans(X,T).$ By Corollary \ref{coro-2} and Corollary \ref{coro-1}, $Y$ is residual in $X$ and $V_T(x)\cap \mathcal{O}_{T}\neq\emptyset$ for any $x\in Y.$  Then we have item (3).  \qed

If further $\sharp \mathcal{M}_{Trans}>1,$ then by Corollary \ref{coro-1} and Theorem \ref{thm-2}, one has $\mathcal{C}_{\mathcal{M}_{Trans}}\subset \mathcal{R}.$ For any $f\in \mathcal{R},$ $I(f,T)$ is residual in $X.$ Then $I(f,T)\cap Trans(X,T)$ is residual in $X.$ Take $x\in I(f,T)\cap Trans(X,T).$ Then $\inf\limits_{\mu\in \mathcal{M}_{Trans}}\int f d\mu\leq \inf\limits_{\mu\in V_T(x)}\int f d\mu\leq \underline{f}(x)<\overline{f}(x)\leq \sup\limits_{\mu\in V_T(x)}\int f d\mu\leq \sup\limits_{\mu\in \mathcal{M}_{Trans}}\int f d\mu.$ This means $f\in \mathcal{C}_{\mathcal{M}_{Trans}}$ and thus $\mathcal{C}_{\mathcal{M}_{Trans}}= \mathcal{R}.$ So by Theorem \ref{thm-2}, $\mathcal{R}$ is open and dense in $C(X),$ and $Y\subset\bigcap\limits_{f\in\mathcal{R}} I(f,T).$   \qed

\subsection{Some facts} In this subsection, we discuss some independent facts.
\begin{lemma}{\label{lemma-7}}
	Suppose that $(X,T)$ is a dynamical system with $Trans(X,T)\neq\emptyset$, $f\in C(X)$, then $f\in\mathcal{R}$ if and only if $I(f,T)\cap Trans(X,T)\neq\emptyset$.
\end{lemma}
\begin{proof}
	When there is no isolated point in $X,$ $Trans(X,T)$ is residual in $X.$ If $I(f,T)\cap Trans(X,T)=\emptyset$, then $I(f,T)$ is of the first category and $f\notin\mathcal{R}$. When there is an isolated point $x_0$ in $X,$ it's easy to see that $x_0\in Trans(X,T).$ Then if $f\in\mathcal{R},$ one has $I(f,T)\cap Trans(X,T)\supset I(f,T)\cap \{x_0\}\neq\emptyset$ since $\{x_0\}$ is open in $X.$ So $f\in\mathcal{R}$ implies $I(f,T)\cap Trans(X,T)\neq\emptyset$.
	
	If $I(f,T)\cap Trans(X,T)\neq\emptyset$, we take a point $x\in I(f,T)\cap Trans(X,T)$, suppose that $$\liminf\limits_{n\to+\infty}\frac1n\sum\limits_{i=0}^{n-1}f(T^ix)<\alpha<\beta<\limsup\limits_{n\to+\infty}\frac1n\sum\limits_{i=0}^{n-1}f(T^ix).$$ Let$$E=\bigcap\limits_{j=1}^{+\infty }{\left( \bigcup\limits_{n=j}^{+\infty }{\{y\in X\mid\frac{1}{n}\sum\limits_{i=0}^{n-1}{f({{T}^{i}}y)<\alpha }\}}\bigcap \bigcup\limits_{n=j}^{+\infty }{\{y\in X\mid\frac{1}{n}\sum\limits_{i=0}^{n-1}{f({{T}^{i}}y)>\beta }\}} \right)}.$$Then $orb(x,T)\subset E\subset I(f,T)$, which means that $I(f,T)$ is residual and $f\in\mathcal{R}$.	
\end{proof}
We denote $\mathcal{U}_x:=\{f\in C(X)\mid x\in I(f,T)\}$. 

\begin{lemma}{\label{lemma-8}}
	Suppose that $(X,T)$ is a dynamical system. If $\mathcal{U}_x\neq\emptyset$, then $\mathcal{U}_x$ is an open and dense set.
\end{lemma}
\begin{proof}
	For any $g\in\mathcal{U}_x$, suppose that
	$$\liminf\limits_{n\to+\infty}\frac1n\sum\limits_{i=0}^{n-1}g(T^ix)<\alpha_g<\beta_g<\limsup\limits_{n\to+\infty}\frac1n\sum\limits_{i=0}^{n-1}g(T^ix).$$
	For any $\tilde{g}\in C(X)$ with
	$||\tilde{g}-g||<\frac{\beta_g-\alpha_g}{3}$, we have $x\in I(\tilde{g},T)$ and $\tilde{g}\in\mathcal{U}_x$, which means that $\mathcal{U}_x$ is open. Since $\mathcal{U}_x\neq\emptyset$, we can choose $f\in\mathcal{U}_x$. For any $h\in C(X)$ with $x\notin I(h,T)$, we have
	$h+\frac{1}{n}f\in\mathcal{U}_x$. Therefore, $\mathcal{U}_x$ is dense in $C(X)$.
\end{proof}

We denote $\mathcal{C}$, $\mathcal{D}$
the set of real continue functions whose irregular set is nonempty, dense 
in $X$ respectively:
\[
\begin{split}
	\mathcal{C}&:=\{f\in C(X)\mid I(f,T)\neq\emptyset\},\\
	\mathcal{D}&:=\{f\in C(X)\mid I(f,T)\text{ is dense in } X\}.
\end{split}
\]
Obviously, $\mathcal{R}\subset\mathcal{D}\subset\mathcal{C}$. 
For any $f\in C(X),$ $r>0$ and $\alpha,\beta\in\mathbb{R}$ we denote
\[
\begin{split}
	E_r^f&:=\{x\in X\mid \limsup\limits_{n\to+\infty}\frac1n\sum\limits_{i=0}^{n-1}f(T^ix)-\liminf\limits_{n\to+\infty}\frac1n\sum\limits_{i=0}^{n-1}f(T^ix)\geq r \},\\
	\overline{E}_\beta^f&:=\{x\in X\mid \limsup\limits_{n\to+\infty}\frac1n\sum\limits_{i=0}^{n-1}f(T^ix)\geq\beta\},\\
	\underline{E}_\alpha^f&:=\{x\in X\mid \liminf\limits_{n\to+\infty}\frac1n\sum\limits_{i=0}^{n-1}f(T^ix)\leq\alpha\}.
\end{split}
\]
Then for any $r>0$ and $\alpha,\beta\in\mathbb{R}$ with $\alpha<\beta$ we define
\[
\begin{split}
	\mathcal{C}_r^*&:=\{f\in C(X)\mid E_r^f\neq\emptyset\},\\
	\mathcal{D}_r^*&:=\{f\in C(X)\mid E_r^f\text{ is dense in } X\},\\
	\mathcal{R}_r^*&:=\{f\in C(X)\mid E_r^f\text{ is residual in } X\},\\
	\mathcal{S}_{\alpha,\beta}&:=\{f\in C(X)\mid \overline{E}_\beta^f \text{ and } \underline{E}_\alpha^f \text{ are both dense in } X\},
\end{split}
\]
and $\mathcal{C}^*=\bigcup\limits_{r>0}\mathcal{C}_r^*,$ $\mathcal{D}^*=\bigcup\limits_{r>0}\mathcal{D}_r^*,$ $\mathcal{R}^*=\bigcup\limits_{r>0}\mathcal{R}_r^*,$ $\mathcal{S}=\bigcup\limits_{\beta>\alpha}\mathcal{S}_{\alpha,\beta}.$

\begin{proposition}\label{prop-4}
	Suppose that $(X,T)$ is a dynamical system, then
	\begin{enumerate}
		\item if $\mathcal{C}\neq\emptyset$, then $\mathcal{C}$ is an open and dense set in $C(X)$ and there is an open and dense set $\mathcal{U}_1$ in $C(X)$ such that $\bigcap\limits_{f\in\mathcal{U}_1} I(f,T)$ is nonempty.
		\item if $\mathcal{D}\neq\emptyset$, then $\mathcal{D}$ is residual in $C(X)$ and there is a residual set $\mathcal{U}_2$ in $C(X)$ such that $\bigcap\limits_{f\in\mathcal{U}_2} I(f,T)$ is dense in $X$.
		\item $\mathcal{C}^*,$ $\mathcal{D}^*,$ $\mathcal{R}^*$ and $\mathcal{S}$ satisfy the following properties:
		\begin{itemize}
			\item[(3.1)] $\mathcal{S}\subset \mathcal{D}^*=\mathcal{R}^*\subset \mathcal{R}\subset \mathcal{D}\subset \mathcal{C}=\mathcal{C}^*=\bigcup\limits_{x\in X}\mathcal{U}_x;$
			\item[(3.2)] $\mathcal{D}_r^*,$ $\mathcal{R}_r^*,$ $\mathcal{S}_{\alpha,\beta}$ are closed for any $r>0$ and $\alpha,\beta\in\mathbb{R}$ with $\alpha<\beta$, and $\mathcal{D}^*,$ $\mathcal{R}^*,$ $\mathcal{S}$ are open;
			\item[(3.3)] if $\mathcal{D}^*\neq\emptyset,$ then $\bigcap\limits_{f\in\mathcal{D}^*} I(f,T)$ is residual in $X;$
			\item[(3.4)] if $G^{K}$ is dense in $X$ for some $K\subset \mathcal{M}(X,T),$ then $\mathcal{C}_{K}\subset \mathcal{D}^*=\mathcal{R}^*\subset \mathcal{R}$ and thus $\mathcal{D}^*$ is open and dense in $C(X).$
		\end{itemize}
		\item if $(X,T)$ is transitive and $\mathcal{R}\neq\emptyset$, then $\mathcal{R}=\mathcal{D}^*=\mathcal{R}^*=\mathcal{S}$ is an open and dense set in $C(X)$ and $\bigcap\limits_{f\in\mathcal{R}} I(f,T)$ is residual in $X$.
	\end{enumerate}
\end{proposition}
\begin{proof}
	(1) Since $\mathcal{C}=\bigcup\limits_{x\in X}\mathcal{U}_x$ and $\mathcal{C}\neq\emptyset$, by Lemma \ref{lemma-8}, we have $\mathcal{C}$ is an open and dense set in $C(X)$. Choose $f\in\mathcal{C}$ and $x\in I(f,T)$, let $\mathcal{U}_1=\mathcal{U}_x$, we have $\mathcal{U}_1$ is an open and dense set and $\bigcap\limits_{f\in\mathcal{U}_1} I(f,T)$ is nonempty.
	
	(2) Since $\mathcal{D}\neq\emptyset$ and $X$ is a compact metric space, we can choose $f\in\mathcal{D}$ and a countable set $\{x_1,x_2,\cdots\}\subset I(f,T)$, which is dense in $X$. Then $\bigcap\limits_{i=1}^{+\infty}\mathcal{U}_{x_i}\subset\mathcal{D}$, which means that $\mathcal{D}$ is residual in $C(X)$ and there is a residual set $\mathcal{U}_2:=\bigcap\limits_{i=1}^{+\infty}\mathcal{U}_{x_i}$ in $C(X)$
	such that $\bigcap\limits_{f\in\mathcal{U}_2} I(f,T)$ is dense in $X$.

	(3) (3.1) We only show $\mathcal{S}\subset\mathcal{R}^*$ and $\mathcal{D}^*=\mathcal{R}^*.$
	
	For any $f\in \mathcal{S}$ there are $\alpha<\beta$ such that $\overline{E}_\beta^f$ and $\underline{E}_\alpha^f$ are both dense in $X.$ Since $\overline{E}_\beta^f=\cap_{k\in\mathbb{N^{+}}}\cap_{N\in\mathbb{N^{+}}}\cup_{n\geq N}\left\{x\in X\mid \frac1n\sum_{i=0}^{n-1}f(T^ix)>\beta-\frac{1}{k} \right\}$ and $\underline{E}_\alpha^f=\cap_{k\in\mathbb{N^{+}}}\cap_{N\in\mathbb{N^{+}}}\cup_{n\geq N}\left\{x\in X\mid \frac1n\sum_{i=0}^{n-1}f(T^ix)<\alpha+\frac{1}{k} \right\}$ are two $G_\delta$ sets,
	then $\overline{E}_\beta^f$ and $\underline{E}_\alpha^f$ are both residual in $X.$ Thus $\overline{E}_\beta^f\cap \underline{E}_\alpha^f$ is residual in $X$ and $\overline{E}_\beta^f\cap \underline{E}_\alpha^f\subset E_{\beta-\alpha}^f$ which imply $f\in \mathcal{R}^*.$ So $\mathcal{S}\subset\mathcal{R}^*.$
	
	Since $E_r^f=\cap_{k\in\mathbb{N^{+}}}\cap_{N\in\mathbb{N^{+}}}\cup_{n\geq N}\cup_{m\geq N}\left\{x\in X\mid |\frac1n\sum_{i=0}^{n-1}f(T^ix)-\frac1m\sum_{i=0}^{m-1}f(T^ix)|>r-\frac{1}{k} \right\}$ is a $G_\delta$ set, then $E_r^f$ is residual in $X$ if and only if $E_r^f$ is dense in $X.$ So $\mathcal{D}_r^*=\mathcal{R}_r^*$ for any $r>0$ and thus $\mathcal{D}^*=\mathcal{R}^*.$
	
	(3.2) For any  $\{f_m\}_{m=1}^{+\infty}\subset \mathcal{R}_r^*$ with $\lim\limits_{m\to\infty}||f_m-f||=0$ for some $f\in C(X),$ we show that $f\in \mathcal{R}_r^*$ which implies $\mathcal{R}_r^*$ is closed. By definition of $\mathcal{R}_r^*,$ $\cap_{m=1}^{+\infty}E_r^{f_m}$ is residual in $X.$ For any $x\in \cap_{m=1}^{+\infty}E_r^{f_m}$ and any $m\in\mathbb{N^+},$ one has $\limsup\limits_{n\to+\infty}\frac1n\sum\limits_{i=0}^{n-1}f_m(T^ix)-\liminf\limits_{n\to+\infty}\frac1n\sum\limits_{i=0}^{n-1}f_m(T^ix)\geq r.$ Given an arbitrary $\varepsilon>0,$ there is $N\in\mathbb{N^{+}}$ such that $||f_m-f||<\varepsilon$ for any $m\geq N.$ Then $\frac1n\sum\limits_{i=0}^{n-1}f_m(T^ix)-\varepsilon<\frac1n\sum\limits_{i=0}^{n-1}f(T^ix)<\frac1n\sum\limits_{i=0}^{n-1}f_m(T^ix)+\varepsilon,$ this imlpies $\limsup\limits_{n\to+\infty}\frac1n\sum\limits_{i=0}^{n-1}f_m(T^ix)-\varepsilon\leq \limsup\limits_{n\to+\infty}\frac1n\sum\limits_{i=0}^{n-1}f(T^ix)$ and $\liminf\limits_{n\to+\infty}\frac1n\sum\limits_{i=0}^{n-1}f(T^ix)\leq \liminf\limits_{n\to+\infty}\frac1n\sum\limits_{i=0}^{n-1}f_m(T^ix)+\varepsilon.$ So $\limsup\limits_{n\to+\infty}\frac1n\sum\limits_{i=0}^{n-1}f(T^ix)-\liminf\limits_{n\to+\infty}\frac1n\sum\limits_{i=0}^{n-1}f(T^ix)\geq r-2\varepsilon.$ By arbitrariness of $\varepsilon,$ we have $f\in \mathcal{R}_r^*.$
	
	Note that $\overline{E}_\beta^f$ and $\underline{E}_\alpha^f$ are two $G_\delta$ sets, $\mathcal{S}_{\alpha,\beta}$ are closed using similar method.
	Since $\mathcal{D}_r^*=\mathcal{R}_r^*$ by the proof of item (3.1), $\mathcal{D}_r^*$ is closed.
	
	For any $f\in C(X)$ and any $s>0,$ denote $B(f,s):=\{g\in C(X)\mid ||g-f||<s\}.$ Then for any $g\in B(f,s),$ one has $\frac1n\sum\limits_{i=0}^{n-1}f(T^ix)-s<\frac1n\sum\limits_{i=0}^{n-1}g(T^ix)<\frac1n\sum\limits_{i=0}^{n-1}f(T^ix)+s,$ this imlpies $\limsup\limits_{n\to+\infty}\frac1n\sum\limits_{i=0}^{n-1}f(T^ix)-s\leq \limsup\limits_{n\to+\infty}\frac1n\sum\limits_{i=0}^{n-1}g(T^ix)$ and $\liminf\limits_{n\to+\infty}\frac1n\sum\limits_{i=0}^{n-1}g(T^ix)\leq \liminf\limits_{n\to+\infty}\frac1n\sum\limits_{i=0}^{n-1}f(T^ix)+s.$ Given $f\in \mathcal{R}_r^*$ for some $r>0,$ then for any $g\in B(f,\frac{r}{3})$ and $x\in E_r^f$ one has $\limsup\limits_{n\to+\infty}\frac1n\sum\limits_{i=0}^{n-1}g(T^ix)-\liminf\limits_{n\to+\infty}\frac1n\sum\limits_{i=0}^{n-1}g(T^ix)\geq \frac{r}{3}$ which implies $E_r^f\subset E_{\frac{r}{3}}^g$ and $B(f,\frac{r}{3})\subset \mathcal{R}_{\frac{r}{3}}^*.$ So $\mathcal{R}^*=\mathcal{D}^*$ is open. Given $f\in \mathcal{S}_{\alpha,\beta}$ for some $\alpha,\beta\in\mathbb{R}$ with $\alpha<\beta,$ then for any $g\in B(f,\frac{\beta-\alpha}{3})$ and $x\in \overline{E}_\beta^f\cap\underline{E}_\alpha^f$ one has $\limsup\limits_{n\to+\infty}\frac1n\sum\limits_{i=0}^{n-1}g(T^ix)\geq \beta-\frac{\beta-\alpha}{3}$ and $\liminf\limits_{n\to+\infty}\frac1n\sum\limits_{i=0}^{n-1}g(T^ix)\leq \alpha+\frac{\beta-\alpha}{3}$ which implies $B(f,\frac{\beta-\alpha}{3})\subset \mathcal{S}_{\frac{2\beta+\alpha}{3},\frac{\beta+2\alpha}{3}}.$ So $\mathcal{S}$ is open.
	
	(3.3) By the proof if item (3.2), for any $f\in \mathcal{D}^*$ there is $r_f>0$ such that $f\in \mathcal{D}_{r_f}^*,$ $B(f,\frac{r_f}{3})\subset \mathcal{D}_{\frac{r_f}{3}}^*$ and $E_{r_f}^f \subset E_{\frac{r_f}{3}}^g$ for any $g\in B(f,\frac{r_f}{3}).$ Since $\mathcal{D}^*$ is a separable metric space and $\mathcal{D}^*=\bigcup\limits_{f\in \mathcal{D}^*}B(f,\frac{r_f}{3})$,  we can find a sequence $\{f_i\}_{i=1}^{+\infty}$ such that $\mathcal{D}^*=\bigcup\limits_{i=1}^{+\infty}B(f_i,\frac{r_{f_i}}{3})$. Then $$\bigcap\limits_{f\in\mathcal{D}^*} I(f,T)=\bigcap\limits_{i=1}^{+\infty}\bigcap\limits_{f\in B(f_i,\frac{r_{f_i}}{3})} I(f,T)\supset\bigcap\limits_{i=1}^{+\infty}E_{r_{f_i}^{f_i}}.$$ Therefore,  $\bigcap\limits_{f\in\mathcal{D}^*} I(f,T)$ is residual in $X.$
	
	(3.4) 
	Note that $\limsup\limits_{n\to+\infty}\frac1n\sum\limits_{i=0}^{n-1}f(T^ix)-\liminf\limits_{n\to+\infty}\frac1n\sum\limits_{i=0}^{n-1}f(T^ix)\geq \sup_{\nu\in K}\int fd\mu-\inf_{\nu\in K}\int fd\mu$ for any $f\in \mathcal{C}_{K}$ and $x\in G^{K}.$
	If $G^{K}$ is dense in $X,$ then $\mathcal{C}_{K}\subset \mathcal{D}^*$ and thus $\mathcal{D}^*$ is open and dense in $C(X)$ by item (3.2) and Theorem  \ref{thm-2}.
	
	(4) (i) By Lemma \ref{lemma-7}, we have that $\mathcal{R}=\bigcup\limits_{x\in Trans(X,T)}\mathcal{U}_x$. Since $\mathcal{R}\neq\emptyset$, by Lemma \ref{lemma-8}, we have that $\mathcal{R}$ is an open and dense set. For any $f\in \mathcal{R},$ by Lemma \ref{lemma-7} one has $I(f,T)\cap Trans(X,T)\neq\emptyset.$ Take $x_0\in I(f,T)\cap Trans(X,T).$  Then $orb(x_0,T)\subset \overline{E}_\beta^f\cap \underline{E}_\alpha^f$ where $\beta=\limsup\limits_{n\to+\infty}\frac1n\sum\limits_{i=0}^{n-1}f(T^ix_0)$ and $\alpha=\liminf\limits_{n\to+\infty}\frac1n\sum\limits_{i=0}^{n-1}f(T^ix_0).$ So $f\in \mathcal{S}_{\alpha,\beta}$ and $\mathcal{R}\subset\mathcal{S}.$ According to $\mathcal{S}\subset \mathcal{D}^*=\mathcal{R}^*\subset \mathcal{R}$ we have $\mathcal{R}=\mathcal{D}^*=\mathcal{R}^*=\mathcal{S}.$
	
	(ii) By item (3.3) and (4)(i), $\bigcap\limits_{f\in\mathcal{R}} I(f,T)=\bigcap\limits_{f\in\mathcal{D}^*} I(f,T)$ is residual in $X$.
\end{proof}

\section{Proof of Theorem \ref{maintheorem-3}}\label{section-4}
\subsection{Proof of Theorem \ref{maintheorem-3}(1)}

Before giving the proof of Theorem \ref{maintheorem-3}(1), we need one lemma.
We define $C^*(X):=\{f:X\to \mathbb{R}\mid \text{ there is } Y\subset X \text{ such that } f|_Y \text{ is continuous and } Y \text{ is residual in }X\}.$ There are many functions in $C^*(X)$ which may not be continuous. For example, every upper semicontinuous function is in $C^*(X).$ By \cite[Theorem 7.3]{Oxtoby1980}, if $f$ can be represented as the limit of an everywhere convergent sequence of continuous functions, then $f$ is continuous except at a set of points of first category. This implies that for any differentiable function $f,$ its derivative $f'$ is in $C^*(X).$

\begin{lemma}\label{prop-AE}
	Suppose that $(X,T)$ is a transitive dynamical system. Given a subadditive sequence $\mathfrak{F}=\{f_n\}_{n=1}^{\infty}$ with $f_n\in C^*(X)$ for any $n\in\mathbb{N},$ if $T$-invarinat subset $Y\subset X$ is residual in $X$ and $f_n|_Y$ is continuous for any $n\in\mathbb{N},$ then $E_{\mathfrak{F}_{\sup}^T}=\{y\in Trans(X,T)\cap Y\mid \limsup\limits_{n\to+\infty}\frac{1}{n}f_n(y)=\mathfrak{F}_{\sup}^T(Y) \}$ is residual in $X$ where $\mathfrak{F}_{\sup}^T(Y)=\sup\limits_{x\in Trans(X,T)\cap Y}\limsup\limits_{n\to+\infty}\frac{1}{n}f_n(x).$
\end{lemma}
\begin{proof}
	For any $\beta<\mathfrak{F}_{\sup}^T,$ denote $E_\beta=\cap_{N=0}^{\infty}\cup_{n=N}^{\infty}\{y\in Trans(X,T)\cap Y\mid \frac1nf_n(y)>\beta\}.$ By $\beta<\mathfrak{F}_{\sup}^T,$ there exists $x_0\in Trans(X,T)\cap Y$ such that $x_0\in E_\beta.$ Since $f_n|_Y$ is continous, $\{y\in Trans(X,T)\cap Y\mid \frac1nf_n(y)>\beta\}$ is open in $Trans(X,T)\cap Y$. Then $E_\beta$ is a nonempty $G_\delta$ set in $Trans(X,T)\cap Y$. By $f_{n+1}(x)\leq f_n(Tx)+f_1(x),$
	we have $\limsup\limits_{n\to+\infty}\frac1nf_n(Tx)\geq\limsup\limits_{n\to+\infty}\frac1nf_n(x)$ for any $x\in X.$
	Then $T^nx_0\in E_\beta$ for any $n\in\mathbb{N}$ by $x_0\in E_\beta.$ So $E_\beta$ is a dense $G_\delta$ set in $Trans(X,T)\cap Y$.
	
	Choose $\{\beta_n\}_{n=1}^{\infty}\subset \mathbb{R}$ with $\beta_n<\mathfrak{F}_{\sup}^T$ and $\lim\limits_{n\to\infty}\beta_n=\mathfrak{F}_{\sup}^T.$ Then $E_{\beta_n}$ is a dense $G_\delta$ set in $Trans(X,T)\cap Y$, and $\cap_{n=1}^{\infty}E_{\beta_n}\subset E_{\mathfrak{F}_{\sup}^T}.$ So $E_{\mathfrak{F}_{\sup}^T}$ is residual in $Trans(X,T)\cap Y.$ Since $Trans(X,T)\cap Y$ is residual in $X,$ we have $E_{\mathfrak{F}_{\sup}^T}$ is residual in $X.$
\end{proof}
Now, we prove Theorem \ref{maintheorem-3}(1).

\noindent\textbf{Proof of Theorem \ref{maintheorem-3}(1)}:
Let $Y=X$ in Lemma \ref{prop-AE}, we complete the proof of Theorem \ref{maintheorem-3}(1). \qed

\subsection{Proof of Theorem \ref{maintheorem-3}(2)}
Before giving the proof of Theorem \ref{maintheorem-3}(2), we need one lemma.
\begin{lemma}\label{lemma-AC}
	Suppose that $(X,T)$ is a transitive dynamical system. Given a subadditive sequence $\mathfrak{F}=\{f_n\}_{n=1}^{\infty}\subset C(X),$  denote $\mathfrak{F}_{\inf}^T=\inf\limits_{x\in Trans(X,T)}\liminf\limits_{n\to+\infty}\frac{1}{n}f_n(x),$
	if for any $x\in X$ there is $C_n(x)\in\mathbb{R}$ with $\lim\limits_{n\to\infty}\frac{1}{n}C_n(x)=0$ such that $f_{n}(Tx)\leq f_{n+1}(x)+C_n(x).$ then $\tilde{E}_{\mathfrak{F}_{\inf}^T}=\{y\in Trans(X,T)\mid \liminf\limits_{n\to+\infty}\frac{1}{n}f_n(y)=\mathfrak{F}_{\inf}^T \}$ is residual in $X.$
\end{lemma}
\begin{proof}
	For any $\alpha>\mathfrak{F}_{\inf}^T,$ denote $\tilde{E}_\alpha=\cap_{N=0}^{\infty}\cup_{n=N}^{\infty}\{y\in Trans(X,T)\mid \frac1nf_n(y)<\alpha\}.$ By $\alpha>\mathfrak{F}_{\inf}^T,$ there exists $x_0\in Trans(X,T)$ such that $x_0\in \tilde{E}_\alpha.$ Since $f_n$ is continous, $\{y\in Trans(X,T)\mid \frac1nf_n(y)<\alpha\}$ is open in $Trans(X,T)$. Then $\tilde{E}_\alpha$ is a nonempty $G_\delta$ set in $Trans(X,T)$. By $f_{n}(Tx)\leq f_{n+1}(x)+C_n(x),$
	we have $\liminf\limits_{n\to+\infty}\frac1nf_n(Tx)\leq\liminf\limits_{n\to+\infty}\frac1nf_n(x)$ for any $x\in X.$
	Then $T^nx_0\in \tilde{E}_\alpha$ for any $n\in\mathbb{N}$ by $x_0\in \tilde{E}_\alpha.$ So $\tilde{E}_\alpha$ is a dense $G_\delta$ set in $Trans(X,T)$.
	
	Choose $\{\alpha_n\}_{n=1}^{\infty}\subset \mathbb{R}$ with $\alpha_n>\mathfrak{F}_{\inf}^T$ and $\lim\limits_{n\to\infty}\alpha_n=\mathfrak{F}_{\inf}^T.$ Then $\tilde{E}_{\alpha_n}$ is a dense $G_\delta$ set in $Trans(X,T)$, and $\cap_{n=1}^{\infty}\tilde{E}_{\alpha_n}\subset \tilde{E}_{\mathfrak{F}_{\inf}^T}.$ So $\tilde{E}_{\mathfrak{F}_{\inf}^T}$ is residual in $Trans(X,T).$ Since $Trans(X,T)$ is residual in $X,$ we have $\tilde{E}_{\mathfrak{F}_{\inf}^T}$ is residual in $X.$
\end{proof}

Now, we prove Theorem \ref{maintheorem-3}(2).

\noindent\textbf{Proof of Theorem \ref{maintheorem-3}(2)}:
Note that $\|A_n(Tx)\|\leq \|A_{n+1}(x)\|\|A(x)^{-1}\|$ and $\{\log\|A_n(x)\|\}_{n=1}^{\infty}$ is a subadditive sequence. Using Lemma \ref{prop-AE} and Lemma \ref{lemma-AC}, we have Theorem \ref{maintheorem-3}(2). \qed

\subsection{Proof of Theorem \ref{maintheorem-3}(3)}
Before giving the proof of Theorem \ref{maintheorem-3}(3), we need one proposition.
\begin{proposition}\label{prop-1}
	Suppose that $(X,T)$ is a dynamical system. If $\mathcal{R}\neq\emptyset,$ then for any $d\geq 1,$ one has $\emptyset\neq\mathcal{R}_d\subset\mathcal{C}_d$ and  $\mathcal{C}_d$ is dense in $C(X,GL(d,\mathbb{R})).$ In particular, if further $(X,T)$ is transitive, then $\mathcal{R}_d$ is dense in $C(X,GL(d,\mathbb{R})).$
\end{proposition}
\begin{proof}
	For any $f\in C(X),$ define $A^f\in C(X,GL(d,\mathbb{R}))$ by $A^f(x)=e^{f(x)}I_d$ where $I_d$ is $d\times d$ identity matrix. Note that
	\begin{equation*}
		\begin{split}
			\frac1n\log\|A^f_n(x)\|&=\frac1n\log\|A^f(T^{n-1}x)\cdots A^f(Tx)A^f(x)\|\\
			&=\frac1n\log \prod_{i=0}^{n-1}e^{f(T^ix)}\\
			&=\frac1n\sum_{i=0}^{n-1}f(T^ix),
		\end{split}
	\end{equation*}
	one has $LI_{A^f}=I(f,T).$ Then $A^f\in \mathcal{R}_d$ for any $f\in \mathcal{R}.$  So $\mathcal{R}_d\neq\emptyset.$
	
	Next,  we show that $\mathcal{C}_d$ is dense in $C(X,GL(d,\mathbb{R})).$ Take $f\in \mathcal{R}.$ For any $A\in C(X,GL(d,\mathbb{R}))$ and $k\in\mathbb{N},$  let $A_{(k)}=A\cdot A^{\frac1kf}=e^{\frac1kf}A.$ Then $\|A_{(k)}-A\|=\sup\limits_{x\in X}\|A_{(k)}(x)-A(x)\|=\sup\limits_{x\in X}|e^{\frac1kf(x)}-1|\|A(x)\|\leq \max\{e^{\frac1k\|f\|}-1,1-e^{-\frac1k\|f\|}\}\|A\|$ which implies $\lim\limits_{k\to\infty}A_{(k)}=A.$ Note that
	\begin{equation*}
		\begin{split}
			\frac1n\log\|(A_{(k)})_n(x)\|&=\frac1n\log\prod_{i=0}^{n-1}e^{\frac1kf(T^ix)}\|A(T^{n-1}x)\cdots A(Tx)A(x)\|\\
			&=\frac1n\log \prod_{i=0}^{n-1}e^{\frac1kf(T^ix)}+\frac1n\log\|A_n(x)\|\\
			&=\frac1n\sum_{i=0}^{n-1}\frac1kf(T^ix)+\frac1n\log\|A_n(x)\|,
		\end{split}
	\end{equation*}
	then for $A\in C(X,GL(d,\mathbb{R}))\setminus \mathcal{C}_d,$ one has $I(f,T)\subset LI_{A_{(k)}}$ by $LI_A=\emptyset$ and $I(\frac1kf,T)=I(f,T).$ So $A_{(k)}\in \mathcal{C}_d$, this means $\mathcal{C}_d$ is dense in $C(X,GL(d,\mathbb{R})).$
	
	Finally, if further $(X,T)$ is transitive, then for $A\in C(X,GL(d,\mathbb{R}))\setminus \mathcal{R}_d,$ one can take $x\in I(f,T)\cap Trans(X,T)\cap(X\setminus LI_A)$ since $I(f,T)\cap Trans(X,T)$ is residual in $X$ and $LI_A$ is not residual in $X.$ Then $x\in LI_{A_{(k)}}\cap Trans(X,T).$ So $A_{(k)}\in \mathcal{R}_d$ by Theorem \ref{maintheorem-3}(2), this means $\mathcal{R}_d$ is dense in $C(X,GL(d,\mathbb{R})).$
\end{proof}

Now, we prove Theorem \ref{maintheorem-3}(3).

\noindent\textbf{Proof of Theorem \ref{maintheorem-3}(3)}:
Since $C(X,GL(d,\mathbb{R}))$ is compact, then by Proposition \ref{prop-1} there exist $\{^{i}A\}_{i=1}^{+\infty}\subset C(X,GL(d,\mathbb{R}))$ such that $^{i}A\in \mathcal{R}_d$ for any $i\in\mathbb{N^+}.$ By Proposition \ref{prop-3}, one has $\mathcal{M}_{dense}^*=\mathcal{M}_{Trans}.$ Then $\sharp \mathcal{M}_{dense}^*>1$ by $\sharp \mathcal{M}_{Trans}>1.$ Thus by Theorem \ref{maintheorem-2}
there exist a residual invariant subset $Y'\subset X$ such that for any $x\in Y'$  one has
\begin{enumerate}
	\item [(a)] $x\in Trans(X,T)$ and $V_T(x)=\mathcal{M}_{Trans};$
	\item [(b)] $V_T(x)\cap \mathcal{O}_{T}\neq\emptyset$ provided that $X=M$ is a manifold.
\end{enumerate}
Let $Y=\bigcap_{i=1}^{+\infty}LI_{{~}^{i}A}\cap Y'$ and $\mathfrak{U}=\{^{i}A\}_{i=1}^{+\infty},$ then we complete the proof of Theorem \ref{maintheorem-3}(3). \qed

Next, we prove Corollary \ref{coro-4} and Corollary \ref{coro-6}.

\noindent\textbf{Proof of Corollary \ref{coro-4}}:
Since $\mu_1\neq\mu_2\in \cM^{e}(X,T)$ and $S_{\mu_1}=S_{\mu_2}=X$, one has $\mu_1(Trans(X,f))=\mu_2(Trans(X,f))=1.$ Then there exist $x_1,x_2\in Trans(X,f)$ such that $\chi(x_1)\leq \int \chi(x)d\mu_1<\int \chi(x)d\mu_2\leq \chi(x_2).$ By Theorem \ref{maintheorem-3}(2), we complete the proof.  \qed

\noindent\textbf{Proof of Corollary \ref{coro-6}}:
Since $(X,T)$ is minimal, every point in $X$ is transitive point.
If there is $x_1,x_2\in X$ such that $\chi(x_1)\neq \chi(x_1),$
by Theorem \ref{maintheorem-3}(2), we have $A\in\mathcal{R}_d.$  \qed

\begin{remark}
	There are some systems satisfying Corollary \ref{coro-4}. For examples, (1) $(X,T)$ is minimal and is not uniquely ergodic. $A=e^{f(x)}I_d$ with $f\in \mathcal{R}.$ (2) $(X,T)$ has approximate product property, $C_T(X)=X,$ and is not uniquely ergodic. Then $\{\mu\in\cM^{e}(X,T)\mid S_\mu=X\}$ is residual in $\cM(X,T)$ by Remark \ref{rem-1}(1)(b). Let $A=e^{f(x)}I_d$ with $f\in \mathcal{R}.$
\end{remark}

\section{$m$-$g$-product property and proof of Theorem \ref{maintheorem-4}}\label{section-5}
\subsection{Two lemmas}
First, we give two lemmas which will be used in the proof of Theorem \ref{maintheorem-4}.
\begin{lemma}\label{lemma-5}
	Suppose that $(X,T)$ is a dynamical system. Then for any $x,y\in X,$ any $p,n\in\mathbb{N^+},$ any $\Lambda\subset \Lambda_n$ and any $f\in C(X),$ one has $|\frac{1}{p+n}\sum_{i=0}^{p+n-1}f(T^ix)-\frac{1}{n}\sum_{i=0}^{n-1}f(T^iy)|\leq \frac{p+n-\sharp \Lambda}{p+n}(B-A)+\frac{1}{p+n}|\sum_{i\in\Lambda}(f(T^{i+p}x)-f(T^iy))|,$ where $B=\sup\limits_{x\in X}f(x),$ $A=\inf\limits_{x\in X}f(x).$
\end{lemma}
\begin{proof}
	\[\begin{split}
		&|\frac{1}{p+n}\sum_{i=0}^{p+n-1}f(T^ix)-\frac{1}{n}\sum_{i=0}^{n-1}f(T^iy)|\\
		=&|\frac{1}{p+n}\sum_{i=0}^{p-1}f(T^ix)+\frac{1}{p+n}\sum_{i=0}^{n-1}f(T^{i+p}x)-\frac{1}{p+n}\sum_{i=0}^{n-1}f(T^iy)-\frac{p}{n(p+n)}\sum_{i=0}^{n-1}f(T^iy)|\\
		\leq &|\frac{1}{p+n}\sum_{i=0}^{p-1}f(T^ix)-\frac{p}{n(p+n)}\sum_{i=0}^{n-1}f(T^iy)|+|\frac{1}{p+n}\sum_{i=0}^{n-1}f(T^{i+p}x)-\frac{1}{p+n}\sum_{i=0}^{n-1}f(T^iy)|\\
		\leq &\frac{p}{p+n}(B-A)+\frac{1}{p+n}|\sum_{i=0}^{n-1}(f(T^{i+p}x)-f(T^iy))|\\
		\leq &\frac{p}{p+n}(B-A)+\frac{1}{p+n}|\sum_{i\in \Lambda_n\setminus \Lambda}(f(T^{i+p}x)-f(T^iy))|+\frac{1}{p+n}|\sum_{i\in \Lambda}(f(T^{i+p}x)-f(T^iy))|\\
		\leq &\frac{p+n-\sharp \Lambda}{p+n}(B-A)+\frac{1}{p+n}|\sum_{i\in\Lambda}(f(T^{i+p}x)-f(T^iy))|.
	\end{split}\]
\end{proof}
\begin{lemma}\label{lemma-4}
	Suppose that $(X,T)$ is a dynamical system. Let $V\subset\mathcal{M}(X,T)$ be a nonempty open set, and $\tilde{V}\subset V$ be a nonempty closed set, then there exist $\tau>0$ and $K\in\mathbb{N^+}$ such that if $\nu\in\mathcal{M}(X)$ satisfies $\ |\int f_kd\nu-\int f_kd\mu|<\tau$ for some $\mu\in\tilde{V}$ and any $1\leq k\leq K,$ then $\nu\in V.$
\end{lemma}
\begin{proof}
	Since $\tilde{V}$ is closed and $V$ is open, there is $\tau>0$ such that $\tilde{V}_\tau:=\{\omega\in \mathcal{M}(X,T)\mid \rho(\omega,\tilde{V})\leq \tau\}\subset V.$ Take $K\in\mathbb{N^{+}}$ such that $\sum_{i=K+1}^{+\infty}\frac{1}{2^{i}}<\frac{\tau}{2},$ if $\nu\in\mathcal{M}(X)$ satisfies $\ |\int f_kd\nu-\int f_kd\mu|<\tau$ for some $\mu\in\tilde{V}$ and any $1\leq k\leq K,$ one has $\rho(\nu,\mu)=\sum_{i=1}^{+\infty}\frac{|\int f_id\nu-\int f_id\mu|}{2^{i+1}}\leq \sum_{i=1}^{K}\frac{\tau}{2^{i+1}}+\sum_{i=K+1}^{+\infty}\frac{1}{2^{i}}<\tau.$ So $\nu\in\tilde{V}_\tau\subset V.$
\end{proof}

\subsection{Proof of Theorem \ref{maintheorem-4}(1)}
Before giving the proof of Theorem \ref{maintheorem-4}(1), we give some definitions and lemmas.
\begin{definition}\label{def-1}
	Suppose that $(X,T)$ is a dynamical system. For any $f\in C(X)$, we say $f$ is $\kappa$-controlled for some $0< \kappa< 1$ if
	\begin{equation}
		\sup\limits_{\mu\in\mathcal{M}(X,T)}\int f\mathrm{d\mu}-\inf\limits_{\mu\in\mathcal{M}(X,T)}\int f\mathrm{d\mu}> \kappa(\sup\limits_{x\in X}f(x)-\inf\limits_{x\in X}f(x)).
	\end{equation}
	For any $0<\kappa<1$, we denote $C_{\kappa}(X):=\{f\in C(X) \mid f\text{ is }\kappa\text{-controlled }\}.$
\end{definition}
Next, we show some information about the set $C_\kappa(X)$.

\begin{lemma}{\label{lemma-9}}
	Suppose that $(X,T)$ is a dynamical system. Then for any $0<\kappa<1$,
	\begin{enumerate}
		\item if $(X,T)$ is uniquely ergodic, then $C_{\kappa}(X)=\emptyset$;
		\item if $(X,T)$ is not uniquely ergodic, then
		\begin{enumerate}
			\item if $T=id$, then every non-constant continuous function is in $C_{\kappa}(X)$;
			\item if $T\neq id$, then $C_{\kappa}(X)$ is a nonempty open set, but it is not dense in $C(X)$.
		\end{enumerate}
	\end{enumerate}
\end{lemma}
\begin{proof}
	(1) It is obvious since for any $f\in C(X)$, $\inf\limits_{\mu\in \mathcal{M}(X,T)}\int f\mathrm{d\mu}=\sup\limits_{\mu\in\mathcal{M}(X,T)}\int f\mathrm{d\mu}.$
	
	(2)(a) If $T=id$, then $\mathcal{M}(X,T)=\mathcal{M}(X)$. As a result, for any $f\in C(X)$, $$\sup\limits_{\mu\in\mathcal{M}(X,T)}\int f\mathrm{d\mu}-\inf\limits_{\mu\in\mathcal{M}(X,T)}\int f\mathrm{d\mu}=\sup\limits_{x\in X}f(x)-\inf\limits_{x\in X}f(x).$$ Then
	for any $0<\kappa<1$, every non-constant continuous function is in $C_{\kappa}(X)$.
	
	(b) if $T\neq id$, For any $f\in C(X)$, let $\varphi(f)=\sup\limits_{\mu\in M(X,T)}\int f\mathrm{d\mu}$ and $\psi(f)=\sup\limits_{x\in X}f(x)$, then $$ C_{\kappa}(X)=\{f\in C(X)\mid \varphi(f)+\varphi(-f)>\kappa(\psi(f)+\psi(-f))\}.$$ For any $f_1,f_2,f\in C(X)$, we have $\varphi(f_1+f_2)\leq\varphi(f_1)+\varphi(f_2)$, $\psi(f_1+f_2)\leq\psi(f_1)+\psi(f_2)$, $|\varphi(f)|\leq||f||$, $|\psi(f)|\leq||f||$. Therefore, $|\varphi(f_1)-\varphi(f_2)|\leq||f_1-f_2||$,$|\psi(f_1)-\psi(f_2)|\leq||f_1-f_2||$. By the continuity of $\varphi(f)$ and $\psi(f)$, we have $C_{\kappa}(X)$ is an open set. Let $$ D_{\kappa}(X)=\{f\in C(X)\mid\varphi(f)+\varphi(-f)<\kappa(\psi(f)+\psi(-f))\}.$$ 		
	By the continuity of $\varphi(f)$ and $\psi(f)$, we have $D_{\kappa}(X)$ is an open set. For any $f\in C(X)$, $\varphi(f-f\circ T)=0$. If for any $f\in C(X)$, $f-f\circ T$ is a constant function. By the compactness of $X$, for any $f\in C(X)$, $f=f\circ T$, therefore $T=id$, a contradiction. So, there is a $\tilde{f}\in C(X)$, $\tilde{f}-\tilde{f}\circ T$ is non-constant, then for any $0<\kappa<1$, $\tilde{f}-\tilde{f}\circ T\in D_{\kappa}(X)$. Therefore $ D_{\kappa}(X)$ is a nonempty open set, and $ C_{\kappa}(X)$ is not dense in $C(X)$. Since $(X,T)$ is not uniquely ergodic, there are $\mu_{1},\mu_{2}\in\cM^{e}(X,T)$, $A\in\cB$, such that $\mu_{1}(A)=\mu_{2}(X\setminus A)=1$. By the regularity of $\mu_{1}$ and $\mu_{2}$, there are closed sets $B_1\subset A$, $B_2\subset X\setminus A$, such that $\mu_{1}(B_1)>\frac{1+\kappa}{2}$, $\mu_{2}(B_2)>\frac{1+\kappa}{2}$. Let $f(x)=\frac{d(x,B_2)}{d(x,B_1)+d(x,B_2)}$, then $f\in C(X)$, $0\leq f\leq1$, $f|_{B_1}\equiv1$, $f|_{B_2}\equiv0$. So $\int f\mathrm{d\mu_{1}}\geq\int_{B_1}f\mathrm{d\mu_{1}}=\mu_{1}(B_1)>\frac{1+\kappa}{2}$, $\int f\mathrm{d\mu_{2}}=\int_{X\setminus B_2}f\mathrm{d\mu_{2}}\leq1-\mu_{2}(B_2)<\frac{1-\kappa}{2}$. Then $\int f\mathrm{d\mu_{1}}-\int f\mathrm{d\mu_{2}}>\kappa$, therefore $f\in C_{\kappa}(X)$ and $C_{\kappa}(X)\neq\emptyset$.      		
\end{proof}

    Due to this lemma, we have
\begin{corollary}{\label{coro-5}}
	Suppose that $(X,T)$ has $m$-$g$-product property with $g(\varepsilon,x,y,n)\leq n$ for any $\varepsilon>0$, any $x,y\in X$ and some $n(\varepsilon,x,y)\geq N(\varepsilon,x,y)$. If $(X,T)$ is not uniquely ergodic, then for any $0<\kappa<1$, $C_{\kappa}(X)$ is a nonempty open set, but it is not dense in $C(X)$.
\end{corollary}
\begin{proof}
	Since that $(X,T)$ is not uniquely ergodic, it has at least two points. It can be easily checked that $T\neq id$.
\end{proof}

In fact, we can consider a broader definition of $m$-$g$-product property as following.
\begin{definition}
	Suppose that $(X,T)$ is a dynamical system, $Y\subset X$ is a nonempty $T$-invariant Borel set. We say $(X,T)$ has $m$-$g$-product property from $X$ to $Y$, if for any $\varepsilon>0$, $x_1\in X$, $x_2\in Y$, there exists $N(\varepsilon,x_1,x_2)\in\mathbb{N^{+}}$ and for any positive integer $n\geq N$, there exist $z\in X,$ $p\in\mathbb{N},$ $\Lambda\subseteq \Lambda_n$, such that $0\leq p\leq m(\varepsilon,x_1,x_2,n),$ $|\Lambda_n\setminus \Lambda|\leq g(\varepsilon,x_1,x_2,n),$ and
	$$d(z,x_1)<\varepsilon,\ d(T^{p+j}z,T^jx_2)<\varepsilon\ \text{for any}\ j\in \Lambda.$$
	In particular, when $Y=X$, we say $(X,T)$ has $m$-$g$-product property. If $(X,T)$ has $m$-$g$-product property and is transitive, we say $(X,T)$ is $m$-$g$-transitive. If $m\equiv0$ or $g\equiv0$, we will omit $m$ or $g$.
\end{definition}
\begin{lemma}{\label{lemma-10}}
	Suppose that $(X,T)$ is a dynamical system satisfying $m$-$g$-product  property from $X$ to $Y$. Let $f\in C(X)$.
	If there are $y,y'\in Y$ such that $$\liminf\limits_{n\to+\infty}\frac{m(\varepsilon,x,z,n)+g(\varepsilon,x,z,n)}{m(\varepsilon,x,z,n)+n}<\kappa< \frac{1}{2}$$ for any $\varepsilon>0,$ any $x\in X$, any $z\in\{y,y'\}$ and  $$\lim_{n\to+\infty}\frac{1}{n}\sum_{i=0}^{n-1}f(T^iy')-\lim_{n\to+\infty}\frac{1}{n}\sum_{i=0}^{n-1}f(T^iy)>2\kappa(\sup\limits_{x\in X}f(x)-\inf\limits_{x\in X}f(x)),$$ then $f\in\mathcal{R}$.
\end{lemma}
\begin{proof}
	Denote $\beta=\lim\limits_{n\to+\infty}\frac{1}{n}\sum_{i=0}^{n-1}f(T^iy'),$ $\alpha=\lim\limits_{n\to+\infty}\frac{1}{n}\sum_{i=0}^{n-1}f(T^iy)$ and $B=\sup\limits_{x\in X}f(x),$ $A=\inf\limits_{x\in X}f(x).$ Then $B>A$ and $\frac{\beta-\alpha}{B-A}>2\kappa,$ we can take $\tilde{\kappa}\in\mathbb{R}$ such that $\frac{\beta-\alpha}{B-A}>2\tilde{\kappa}>2\kappa$. Take $\tau\in\mathbb{R}$ such that $0<\tau<(1-\frac{\kappa}{\tilde{\kappa}})\frac{\beta-\alpha}{4}$. For any $n\in\mathbb{N^{+}}$, let
	$$U_{n}=\{x\in X \mid\exists N\geq n,\frac{1}{N}\sum_{i=0}^{N-1}f(T^{i}x)<\frac{\alpha+\beta}{2}-\tau\},$$
	$$V_{n}=\{x\in X \mid\exists N\geq n,\frac{1}{N}\sum_{i=0}^{N-1}f(T^{i}x)>\frac{\alpha+\beta}{2}+\tau\}.$$
	Then $U_{n}$ and $V_{n}$ are open sets and $I(f,T)\supset\bigcap\limits_{n=1}^{+\infty}(U_{n}\cap V_{n}).$ By Baire category theorem, we only need to prove that $U_n$ and $V_n$ are dense in $X.$ Next, we give the proof of the denseness of $U_n.$ The denseness of $V_n$ can be proved in the same way.
	For any $x_0\in X,$ $\delta_1>0$ and $l\in\mathbb{N^{+}}$, we will show that $U_l\cap B(x_0,\delta_1)\neq\emptyset.$ Due to the compactness of $X$, we can find a constant $\delta_2>0,$ such that for any $x_1,x_2\in X$ satisfying
	$d(x_1,x_2)<\delta_2,$ one has $|f(x_1)-f(x_2)|<\frac{\tau}{2}.$ Let $\varepsilon=\min\{\delta_1,\delta_2\}$.

	Since $\liminf\limits_{n\to+\infty}\frac{m(\varepsilon,x_0,y,n)+g(\varepsilon,x_0,y,n)}{m(\varepsilon,x_0,y,n)+n}<\kappa,$ there exist $\{n_l\}_{l=1}^{+\infty}\subset \mathbb{N}$ and $M\in\mathbb{N}^+$ such that for any $n_l\geq M,$ one has
	\begin{equation}\label{equation-1}
		n_l\geq N(\varepsilon,x_0,y),\ n_l\geq l,\ \frac{m(\varepsilon,x_0,y,n_l)+g(\varepsilon,x_0,y,n_l)}{m(\varepsilon,x_0,y,n_l)+n_l}<\kappa.
	\end{equation}
	Since $\alpha=\lim\limits_{n\to+\infty}\frac{1}{n}\sum_{i=0}^{n-1}f(T^iy)$, there exists $n_l\geq M$ such that $|\frac{1}{n_l}\sum_{i=0}^{n_l-1}f(T^iy)-\alpha|<\frac{\tau}{2}.$
	Then by $m$-$g$-product property of $(X,T),$ there is $z\in X$, $0\leq p\leq m(\varepsilon,x_0,y,n_l)$, $\Lambda\subseteq\Lambda_{n_l}$ such that $$|\Lambda_{n_l}\setminus\Lambda|\leq g(\varepsilon,x_0,y,n_l),\ d(z,x_0)<\varepsilon\leq\delta_1,\ d(T^{p+j}z,T^jy)<\varepsilon\leq\delta_2\ \text{for any}\ j\in \Lambda.$$
	Thereby $|f(T^{p+j}z)-f(T^jy)|<\frac{\tau}{2}$, $\forall j\in \Lambda$. Then one has
	\[
	\begin{split}
		\frac{1}{p+n_l}\sum_{i=0}^{p+n_l-1}f(T^iz)-\alpha&=\frac{1}{p+n_l}\sum_{i=0}^{p+n_l-1}f(T^iz)-\frac{1}{n}\sum_{i=0}^{n_l-1}f(T^iy)+\frac{1}{n_l}\sum_{i=0}^{n_l-1}f(T^iy)-\alpha\\
		&\leq\frac{p+n_l-\sharp \Lambda}{p+n_l}(B-A)+\frac{1}{p+n_l}|\sum_{i\in\Lambda}(f(T^{i+p}z)-f(T^iy))|+\frac{\tau}{2}\\
		&\leq\frac{m+g}{m+n_l}(B-A)+\frac{\sharp \Lambda}{p+n_l}\frac{\tau}{2}+\frac{\tau}{2}\\
		&\leq\kappa\frac{\beta-\alpha}{2\tilde{\kappa}}+\tau<\frac{\beta-\alpha}{2}-\tau\\
	\end{split}
	\]
	by Lemma \ref{lemma-5}, (\ref{equation-1}) and $0<\tau<(1-\frac{\kappa}{\tilde{\kappa}})\frac{\beta-\alpha}{4}$. So one has $z\in U_l\cap B(x_0,\delta_1).$ As a result, $U_l$ is dense in $X$.
\end{proof}

Now, we prove Theorem \ref{maintheorem-4}(1).

\noindent\textbf{Proof of Theorem \ref{maintheorem-4}(1)}:
By Corollary \ref{coro-5}, $C_{2\kappa}(X)$ is a nonempty open set. For any $f\in C_{2\kappa}(X),$ one has $\sup\limits_{\mu\in\mathcal{M}^e(X,T)}\int f\mathrm{d\mu}-\inf\limits_{\mu\in\mathcal{M}^e(X,T)}\int f\mathrm{d\mu}=\sup\limits_{\mu\in\mathcal{M}(X,T)}\int f\mathrm{d\mu}-\inf\limits_{\mu\in\mathcal{M}(X,T)}\int f\mathrm{d\mu}> 2\kappa(\sup\limits_{x\in X}f(x)-\inf\limits_{x\in X}f(x))$ by the ergodic decomposition theorem. Since $\mathcal{M}^e(X,T)$ is compact, there exist $\mu_{1},\mu_{2}\in \mathcal{M}^e(X,T)$ such that $\int f\mathrm{d\mu_1}=\sup\limits_{\mu\in\mathcal{M}^e(X,T)}\int f\mathrm{d\mu}$ and $\int f\mathrm{d\mu_2}=\inf\limits_{\mu\in\mathcal{M}^e(X,T)}\int f\mathrm{d\mu}.$ By the Birkhoff ergodic theorem, there exist $y_1,y_2\in X$ such that $\lim\limits_{n\to+\infty}\frac{1}{n}\sum_{i=0}^{n-1}f(T^iy_1)=\int f\mathrm{d\mu_1}$ and $\lim\limits_{n\to+\infty}\frac{1}{n}\sum_{i=0}^{n-1}f(T^iy_2)=\int f\mathrm{d\mu_2}.$ Then one has $\lim\limits_{n\to+\infty}\frac{1}{n}\sum_{i=0}^{n-1}f(T^iy_1)-\lim\limits_{n\to+\infty}\frac{1}{n}\sum_{i=0}^{n-1}f(T^iy_2)>2\kappa(\sup\limits_{x\in X}f(x)-\inf\limits_{x\in X}f(x)).$ Thus $f\in\mathcal{R}$ by Lemma \ref{lemma-10}. According to Lemma \ref{lemma-7},  we have $I(f,T)\cap Trans(X,T)\neq\emptyset$. Take $x_0\in I(f,T)\cap Trans(X,T).$ Then $\sharp V_T(x_0)>1$ which implies $\sharp\mathcal{M}_{Trans}>1.$ Since $\mathcal{M}_{Trans}=\mathcal{M}_{dense}^*$ by Proposition \ref{prop-3}, we obtain $\sharp \mathcal{M}_{dense}^*>1.$   \qed

\subsection{Proof of Theorem \ref{maintheorem-4}(2)}
Before giving the proof of Theorem \ref{maintheorem-4}(2), we give some lemmas.

\begin{lemma}\label{lemma-11}
	For any nonempty open set $V\subset\mathcal{M}(X,T)$, let $$U(N_0,V):=\bigcup_{N>N_{0}}\left\{x \in X \mid \frac{1}{N} \sum_{j=0}^{N-1} \delta_{T^{j} x} \in V\right\}.$$
	If $(X,T)$ satisfies $m$-$g$-product property with $\liminf\limits_{n\to+\infty}\frac{m(\varepsilon,x,y,n)+g(\varepsilon,x,y,n)}{n}=0$ for any $\varepsilon>0,$ any $x,y\in X,$ and there is $y'\in X$ such that $ V_T(y')\subset V$, then $U(N_0,V)$ is an open and dense set in $X$.
\end{lemma}
\begin{proof}
	It is obvious that $U(N_0,V)$ is an open set, we only need to proof that $U(N_0,V)$ is dense in $X$.
	For any $x_0\in X$, $\delta_1>0$, $N_0\in\mathbb{N^+},$ we will show that $U(N_0,V)\cap B(x_0,\delta_1)\neq\emptyset$ where $B(x_0,\delta_1):=\{x\in X\mid d(x,x_0)<\delta_1\}.$
	By Lemma \ref{lemma-4} there are $\tau>0$ and $K\in\mathbb{N^+}$ such that if $\nu\in\mathcal{M}(X)$ satisfies $\ |\int f_kd\nu-\int f_kd\mu|<\tau$ for some $\mu\in V_T(y')$ and any $1\leq k\leq K,$ then $\nu\in V.$
	Due to the compactness of $X$, we can find a constant $\delta_2>0,$ such that for any $x_1,x_2\in X$ satisfying
	$d(x_1,x_2)<\delta_2,$ one has $|f_k(x_1)-f_k(x_2)|<\frac{\tau}{4}$ for any $1\leq k\leq K$. Let $\varepsilon=\min\{\delta_1,\delta_2\}$. Since $\liminf\limits_{n\to+\infty}\frac{m(\varepsilon,x_0,y',n)+g(\varepsilon,x_0,y',n)}{n}=0,$ There exist $\{n_l\}_{l=1}^{+\infty}\subset \mathbb{N}$ and $M\in\mathbb{N^{+}}$ such that for any $n_l\geq M$, we have
	\begin{equation}\label{equation-2}
		\ n_l\geq N(\varepsilon,x_0,y'),\ n_l> N_0,\ \frac{g(\varepsilon,x_0,y',n_l)}{n_l}(B_k-A_k)< \frac{\tau}{4}\ \text{and}\ \frac{m(\varepsilon,x_0,y',n_l)}{n_l}(B_k-A_k)< \frac{\tau}{4},
	\end{equation}
	for any $1\leq k\leq K,$ where $B_k=\sup\limits_{x\in X}f_k(x),$ $A_k=\inf\limits_{x\in X}f_k(x).$
	Then there exist $\mu\in V_T(y')$ and $n_l\geq M$ such that $|\frac{1}{n_l}\sum\limits_{i=0}^{n_l-1}f_k(T^iy')-\alpha_k|<\frac{\tau}{4}$ for any $1\leq k\leq K$, where $\alpha_k=\int f_k\mathrm{d\mu}$.
	By $m$-$g$-product property of $(X,T),$ there is $z\in X$, $p\in\mathbb{N}^{+},$ $\Lambda\subseteq \Lambda_{n_l}:=\{0,\cdots,n_l-1\}$ such that $0\leq p\leq m,$ $|\Lambda_{n_l}\setminus \Lambda|\leq g,$ and $$d(z,x_0)<\varepsilon\leq\delta_1,\ d(T^{p+j}z,T^jy')<\varepsilon\leq\delta_2\ \text{for any}\ j\in \Lambda.$$
	Thereby $|f_k(T^{p+j}z)-f_k(T^jy')|<\frac{\tau}{4}$, $\forall j\in \Lambda$, $\forall 1\leq k\leq K$.
	Then for any $1\leq k\leq K$, by Lemma \ref{lemma-5} we have
	\[
	\begin{split}
		|\frac{1}{p+n_l}\sum_{i=0}^{p+n_l-1}f_k(T^iz)-\alpha_k|\leq &\frac{p+n_l-\sharp \Lambda}{p+n_l}(B_k-A_k)+\frac{1}{p+n_l}|\sum_{i\in\Lambda}(f(T^{i+p}z)-f(T^iy'))|\\
		&+|\frac{1}{n_l}\sum\limits_{i=0}^{n_l-1}f_k(T^iy')-\alpha_k|\\
		\leq & \frac{m}{m+n_l}(B_k-A_k)+\frac{g}{n_l}(B_k-A_k)+\frac{\sharp \Lambda}{p+n_l}\frac{\tau}{4}+\frac{\tau}{4}\\
		\leq & \frac{\tau}{4}+\frac{\tau}{4}+\frac{\tau}{4}+\frac{\tau}{4}=\tau
	\end{split}
	\]
	by (\ref{equation-2}). So one has $\frac{1}{p+n_l}\sum\limits_{i=0}^{p+n_l-1}\delta_{T^iz}\in V$ and $U(N_0,V)\cap B(x_0,\delta_1)\neq\emptyset$. As a result, $U(N_0,V)$ is dense in $X$.
\end{proof}	
\begin{lemma}\label{lemma-12}
	Suppose that $(X,T)$ is a dynamical system. 
	If $(X,T)$ satisfies $m$-$g$-product property with $\liminf\limits_{n\to+\infty}\frac{m(\varepsilon,x,y,n)+g(\varepsilon,x,y,n)}{n}=0$ for any $\varepsilon>0,$ any $x,y\in X$, then $G^{\{\mu\in\mathcal{M}(X,T)\mid G_\mu\neq\emptyset\}}$ is residual in $X;$
\end{lemma}
\begin{proof}
	Let $K=\overline{\{\mu\in\mathcal{M}(X,T)\mid G_\mu\neq\emptyset\}},$ then $K\subset \cM(X,T)$ is a nonempty compact set. We can find open balls $\{V_i\}_{i\in\mathbb{N^+}}$ and $\{U_i\}_{i\in\mathbb{N^+}}$ in $\cM(X)$ such that
	\begin{enumerate}[(a)]
		\item $V_i\subset\overline{V}_i\subset U_i$;
		\item $diam(U_i)\longrightarrow0$;
		\item $V_i\cap K\neq\emptyset$;
		\item each point of $K$ lies in infinitely many $V_i$.
	\end{enumerate}
	Put $P(U_i):=\{x\in X\mid V_T(x)\cap U_i\neq\emptyset\}$, then $\bigcap\limits_{i=1}^{+\infty}P(U_i)= G^{\overline{\{\mu\in\mathcal{M}(X,T)\mid G_\mu\neq\emptyset\}}},$  and $$
	P\left(U_{i}\right) \supset \bigcap_{N_{0}=1}^{+\infty} \bigcup_{N>N_{0}}\left\{x \in X \mid \frac{1}{N} \sum_{j=0}^{N-1} \delta_{T^{j} x} \in V_{i}\right\}
	.$$ Since $U(N_0,V_i)$ is an open set by Lemma \ref{lemma-13}, by Baire category theorem if we have
	$$\forall N_0\in\mathbb{N^+},\ \forall i\in\mathbb{N^+},\ U(N_0,V_i)\text{ is dense in }X
	,$$
	then $G^{\overline{\{\mu\in\mathcal{M}(X,T)\mid G_\mu\neq\emptyset\}}}$ is residual in $X.$
	Since $V_i$ is open, there exists $\mu\in\cM(X,T)\cap V_i$ such that $G_{\mu}\neq\emptyset.$ Take $y\in G_\mu$, one has $V_T(y)=\{\mu\}\subset V_i.$ Then
	by Lemma \ref{lemma-11}, $U(N_0,V_i)$ is an open and dense set in $X$. So $G^{\{\mu\in\mathcal{M}(X,T)\mid G_\mu\neq\emptyset\}}$ is residual in $X,$ since $G^{\{\mu\in\mathcal{M}(X,T)\mid G_\mu\neq\emptyset\}}= G^{\overline{\{\mu\in\mathcal{M}(X,T)\mid G_\mu\neq\emptyset\}}}$ by Lemma \ref{lemma-1}.
\end{proof}

Now, we prove Theorem \ref{maintheorem-4}(2).

\noindent\textbf{Proof of Theorem \ref{maintheorem-4}(2)}:
By the Birkhoff ergodic theorem, we have $\mathcal{M}^e(X,T)\subset \{\mu\in\mathcal{M}(X,T)\mid G_\mu\neq\emptyset\}.$ Then for any $\mu\in  \mathcal{M}^e(X,T),$ $G^{\{\mu\in\mathcal{M}(X,T)\mid G_\mu\neq\emptyset\}}\subset G^{\mu}.$ Since $G^{\{\mu\in\mathcal{M}(X,T)\mid G_\mu\neq\emptyset\}}$ is residual in $X$ by Lemma \ref{lemma-12}, $G^{\mu}$ is also residual in $X.$ So $\mu\in \mathcal{M}_{dense}^*$ and thus $\mathcal{M}^{e}(X,T)\subset\mathcal{M}_{dense}^*.$\qed

\subsection{Proof of Theorem \ref{maintheorem-4}(3)}
Before giving the proof of Theorem \ref{maintheorem-4}(3), we give some lemmas.

From the proof of \cite[Lemma 3.4]{HTW2019}, we obtain the following lemma.
\begin{lemma}\label{lemma-AA}
	Suppose that $(X,T)$ is a dynamical system satisfying approximate product property. Then ergodic measures supported
	on minimal sets are dense in $\cM(X,T)$.
\end{lemma}
Next, we show that almost periodic points is dense in measure center $C_T(X)$ when $(X,T)$ satisfies approximate product property. We recall the definition of almost periodic point. A point $x \in X$ is almost periodic, if for any open neighborhood $U$ of $x$, there exists $N \in \mathbb{N}$ such that $f^k(x) \in U$ for some $k \in [n, n+N]$ for every $n \in \mathbb{N}$. It is well-known that $x$ is almost periodic, if and only if  $\overline{orb(x,T)}$ is a minimal set.
\begin{lemma}\label{measure-center}
	Suppose that $(X,T)$ is a dynamical system satisfying approximate product property. Then $C_T(X)=\overline{AP}.$
\end{lemma}
\begin{proof}
	For any $x\in AP$ and $\mu\in \mathcal{M}(\overline{\text{orb}(x,T)},T),$ we have $x\in S_\mu.$ Then $\overline{AP}\subset C_T(X).$ To see the reverse direction, by Lemma \ref{lemma-AA}, ergodic measures supported
	on minimal sets are dense in $\mathcal{M}(X,T).$ Thus for any $\mu\in\mathcal{M}(X,T),$ there exists $\{\mu_i\}_{i=1}^{\infty}\subset\mathcal{M}(X,T)$ such that $S_{\mu_i}$ is minimal for any $i\in\mathbb{N^{+}}$ and $\lim_{i\to \infty}\mu_i=\mu$ which means $\mu(\overline{\cup_{i=1}^{\infty}S_{\mu_i}})\geq \limsup_{j\to\infty}\mu_j(\overline{\cup_{i=1}^{\infty}S_{\mu_i}})=1.$ In other words, $S_\mu\subset \overline{\cup_{i=1}^{\infty}S_{\mu_i}}\subset\overline{AP}.$  So $C_T(X)=\overline{AP}.$
\end{proof}
\begin{theorem}\label{thm-3}
	Suppose that $(X,T)$ is a dynamical system satisfying approximate product property. If $C_T(X)=X,$ then $(X,T)$ is $m$-$g$-transitive with $\lim\limits_{n\to+\infty}\frac{m(\varepsilon,x,y,n)+g(\varepsilon,x,y,n)}{n}=0$ for any $\varepsilon>0,$ any $x,y\in X.$
\end{theorem}
\begin{proof}
	For any $\varepsilon>0$ and $x_1,x_2\in X,$ by Lemma \ref{measure-center} there is $y\in AP\cap B(x_1,\frac{\varepsilon}{3})$ and $L_0\in\mathbb{N}^+$ such that for any $l\geq 1$, there is $p\in \left[ l,l+L_{0}-1\right] $ such that $f^{p}(y)\in B(y,\frac{\varepsilon}{3}).$ For any $k\in\mathbb{N^{+}},$ let $N_k=N(\frac{\varepsilon}{3},\frac{1}{2^k},\frac{1}{2^k})$ be defined in the definition of approximate product property. We assume that $N_k<N_{k+1}$ for any $k\in\mathbb{N^{+}}.$ Take $k_0\in\mathbb{N^{+}}$ such that $2^{k_0-1}\leq L_0< 2^{k_0}$ and let $N(\varepsilon,x_1,x_2)=\max\{\frac{L_0+2}{1-\frac{L_0}{2^{k_0}}},N_{k_0}\}.$ Then for any $n\geq N(\varepsilon,x_1,x_2),$ there exists $k\geq k_0$ such that $N_k\leq n<N_{k+1}.$ By approximate product property, there is $z_n\in X$ and $p\in\mathbb{N}$ such that $1\leq p\leq \frac{n}{2^k}$ and
	\begin{equation}\label{equation-AB}
		\#\{0\leq j\leq n-1\mid d(T^{j}z_n,T^jy)\geq\frac{\varepsilon}{3}\}\leq\frac{n}{2^k},
	\end{equation}
	and
	\begin{equation}
		\#\{0\leq j\leq n-1\mid d(T^{n+p+j}z_n,T^jx_2)\geq\frac{\varepsilon}{3}\}\leq\frac{n}{2^k}.
	\end{equation}
	Let $\tilde{n}=[L_0(\frac{n}{2^k}+1)]+1.$ Then one has
	\begin{equation}\label{equation-AJ}
		\sharp\{n-\tilde{n}\leq q \leq n-1:d(T^{q}(y),y) <\frac{\varepsilon}{3}\}\geq [\frac{\tilde{n}}{L_{0}}]>[\frac{n}{2^k}+1]>\frac{n}{2^k},
	\end{equation}
	and
	\begin{equation*}
		n-\tilde{n}>n-L_0(\frac{n}{2^k}+1)-2\geq n-L_0(\frac{n}{2^{k_0}}+1)-2\geq 0
	\end{equation*}
	which implies $\tilde{n}\leq n.$
	Together with (\ref{equation-AB}) we get that there is $q_n\in \left [ n-\tilde{n},n-1\right ]$ such that
	\begin{equation}
		d(T^{q_{n}}z_n,T^{q_{n}}y)<\frac{\varepsilon}{3} \ \mathrm{and} \ d(y,T^{q_n}y)<\frac{\varepsilon}{3},
	\end{equation}
	which implies $d(T^{q_n}z_n,x_1)<\varepsilon$. Let $z=T^{q_n}z_n,$ one has
	\begin{equation}
		d(z,x_1)<\varepsilon\ \text{and}\ \#\{0\leq j\leq n-1\mid d(T^{n-q_n+p+j}z,T^jx_2)\geq\frac{\varepsilon}{3}\}\leq\frac{n}{2^k}.
	\end{equation}
	Let $m(\varepsilon,x,y,n)=[L_0(\frac{n}{2^k}+1)]+2+[\frac{n}{2^k}]$ and $g(\varepsilon,x,y,n)=[\frac{n}{2^k}]+1.$ Then $(X,T)$ has $m$-$g$-product property and $\lim\limits_{n\to+\infty}\frac{m(\varepsilon,x,y,n)+g(\varepsilon,x,y,n)}{n}=0$ for any $\varepsilon>0,$ any $x,y\in X.$
	
	By Lemma \ref{measure-center}, $C_T(X)=\overline{AP}=X.$ Since $\{\mu\in\cM^{e}(X,T)\mid S_\mu=C_T(X)\}$ is residual in $\cM(X,T)$ by Remark \ref{rem-1}(1)(b), there is an ergodic measure with full support which implies $(X,T)$ is transitive. So $(X,T)$ is $m$-$g$-transitive.
\end{proof}

Now, we prove Theorem \ref{maintheorem-4}(3).

\noindent\textbf{Proof of Theorem \ref{maintheorem-4}(3)}:
Let $K=\mathcal{M}(X,T),$ then $K$ is a nonempty compact set. We can find open balls $\{V_i\}_{i\in\mathbb{N^+}}$ and $\{U_i\}_{i\in\mathbb{N^+}}$ in $\cM(X)$ such that
\begin{enumerate}[(a)]
	\item $V_i\subset\overline{V}_i\subset U_i$;
	\item $diam(U_i)\longrightarrow0$;
	\item $V_i\cap K\neq\emptyset$;
	\item each point of $K$ lies in infinitely many $V_i$.
\end{enumerate}
Put $P(U_i):=\{x\in X\mid V_T(x)\cap U_i\neq\emptyset\}$, then $\bigcap\limits_{i=1}^{+\infty}P(U_i)= G^{\mathcal{M}(X,T)},$  and $$
P\left(U_{i}\right) \supset \bigcap_{N_{0}=1}^{+\infty} \bigcup_{N>N_{0}}\left\{x \in X \mid \frac{1}{N} \sum_{j=0}^{N-1} \delta_{T^{j} x} \in V_{i}\right\}
.$$ Since $U(N_0,V_i)$ is an open set by Lemma \ref{lemma-13}, by Baire category theorem if we have
$$\forall N_0\in\mathbb{N^+},\ \forall i\in\mathbb{N^+},\ U(N_0,V_i)\text{ is dense in }X
,$$
then $G^{\mathcal{M}(X,T)}$ is residual in $X.$
By Remark \ref{rem-1}(1)(b), $\{\mu\in\cM^{e}(X,T)\mid S_\mu=X\}$ is  residual in $\cM(X,T).$
Since $V_i$ is open and $\{\mu\in\cM^{e}(X,T)\mid S_\mu=X\}$ is residual in $\cM(X,T),$ there exists $\omega\in\cM^{e}(X,T)\cap V_i$ such that $S_\omega=X.$ Then by $\omega\in\cM^{e}(X,T),$  one has $\omega(G_\omega)=1$ which implies $G_\omega$ is dense in $X.$ It's easy to see that $G_\omega\subset U(N_0,V_i)$ which implies $U(N_0,V_i)$ is dense in $X.$ Thus $G^{\mathcal{M}(X,T)}$ is residual in $X.$ Then for any $\mu\in \mathcal{M}(X,T),$ $G^{\mu}$ is also residual in $X$ by $G^{\mathcal{M}(X,T)}\subset G^{\mu}.$ This implies $\mu\in \mathcal{M}_{dense}^*,$ and thus $\mathcal{M}_{dense}^*=\mathcal{M}(X,T).$ Combining with Theorem \ref{thm-3}, we complete the proof of Theorem \ref{maintheorem-4}(3).  \qed

\subsection{Proof of Theorem \ref{maintheorem-4}(4)}
Before giving the proof of Theorem \ref{maintheorem-4}(4), we need some definitions.
\begin{definition}
	A set $A=\{a_1<a_2<\cdots\}\subset\mathbb{N}$ is said to be syndetic if there is $N>0$, such that for any $k\in\mathbb{N^+}$, $a_{k+1}-a_k\leq N$.
\end{definition}

Denote $\mathcal{F}_s$ the set of all syndetic sets. It is well known that for any $x\in X$, $x$ is a minimal point if and only if that $N(x,U)\in\mathcal{F}_s$ for any neighbourhood $U$ of $x$, where $N(x,U):=\{n\in\mathbb{N}\mid T^n(x)\in U\}$.

Now, we prove Theorem \ref{maintheorem-4}(4).

\noindent\textbf{Proof of Theorem \ref{maintheorem-4}(4)}:
	For any $\varepsilon>0$, since $X$ is a compact metric space, there are $x_1,x_2,\cdots,x_k\in X$ such that $\bigcup\limits_{i=1}^{k}B(x_i,\frac{\varepsilon}{2})=X$, where $B(x_i,\frac{\varepsilon}{2}):=\{\tilde{y}\in X\mid d(x_i,\tilde{y})<\frac{\varepsilon}{2}\}$. Then for any $1\leq i\leq k$, $N(x_i,B(x_i,\frac{\varepsilon}{2}))\in\mathcal{F}_s$, Assume that $N(x_i,B(x_i,\frac{\varepsilon}{2}))=\{a_{i,1}<a_{i,2}<\cdots\}$, then there is $N>0$, such that $$\forall1\leq i\leq k,\ \forall j\in\mathbb{N^{+}},\ a_{i,j+1}-a_{i,j}<N.$$ Let $m(\varepsilon)=N$. For any $y\in X,$ any $n\in\mathbb{N^{+}},$ and any $\varepsilon'>0$, denote $$B_n(y,\varepsilon'):=\{\tilde{y}\in X\mid \max_{0\leq i\leq n-1}d(T^i\tilde{y},T^iy)<\varepsilon'\}.$$ For any $x\in X$, there is $1\leq i\leq k$ such that $x\in B(x_i,\frac{\varepsilon}{2})$. Since $(X,T)$ is minimal, we can find $a>\max\{N,a_{i,1}\}$ such that $T^ax_i\in B_n(y,\varepsilon')$. There is $s\leq N$ such that $T^{a-s}x_i\in B(x_i,\frac{\varepsilon}{2})$. Let $z=T^{a-s}x_i$, then $z\in B(x_i,\frac{\varepsilon}{2})$, $d(z,x)<\varepsilon$ and $T^sz\in B_n(y,\varepsilon')$. 
	
	Finally, for any $\mu\in\cM^e(X,T),$ one has $\mu(G_{\mu})=1$ and $S_{\mu}=X.$ Then $G_{\mu}$ is dense in $X.$ This means $\mu\in \mathcal{M}_{dense}^*$ and thus $\mathcal{M}^{e}(X,T)\subset\mathcal{M}_{dense}^*.$
	\qed

\section{Proof of Theorem \ref{maintheorem-5}}\label{section-6}
We will construct the examples on full shift $\{-1,0,1\}^{\mathbb{Z}}$ inspired by the examples gived by Ronnie Pavlov in \cite{Pavlov2016}.
Denote $\mathcal{L}(X)$, the language of a subshift $X$, which is the set of all words which appear in points of $X$. Denote $\mathcal{L}_n(X)$ the set of words in $\mathcal{L}(X)$ with length $n$. For any word $w\in\mathcal{L}(X)$, denote $|w|$ the length of $w$.
Suppose that $m(n)=[\kappa n]+1$: $\mathbb{N^{+}}\to\mathbb{N^{+}}$, where $\kappa>0$, then $\lim\limits_{n\to+\infty}\frac{m(n)}{n}=\kappa$ and $m(n_1+n_2)\leq m(n_1)+m(n_2)$, for any $n_1,n_2\in\mathbb{N^{+}}$.
We define a subshift $X$ of the full shift $\{-1,0,1\}^{\mathbb{Z}}$ with the set of forbidden words $\mathcal{F}(X)$ as following:
\begin{enumerate}
	\item any word $ij$ where $ij<0$;
	\item any word $v_{j+1}0^kv_j\cdots v_2v_1$ where all $v_i\neq0$ and $1\leq k\leq m(j)$.
\end{enumerate}

First, for any $w\in\mathcal{L}(X)$, $u\in\mathcal{L}(X)$, $v\in\mathbb{N^{+}}$ with $v\geq1+m(|u|)$, we will prove that $\tilde{w}=w0^vu\in\mathcal{L}$. It is obvious that $\tilde{w}$ satisfies (1). As for (2), if $\tilde{w}$ contains a word $v_{j+1}0^kv_j\cdots v_2v_1$, then $k\geq v\geq1+m(|u|)$. Since that $j\leq|u|$ and $m(n)$ is nondecreasing, we have $k>m(j)$, which means that $\tilde{w}$ satisfies (2). As a result, $\tilde{w}\in\mathcal{L}(X)$. It can be checked that the subshift $X$ is a $m$-transitive dynamical system with $\lim\limits_{n\to+\infty}\frac{m(\varepsilon,x,y,n)}{n}=\kappa$.

Second, consider the words $w=1$ and $u=-1-1\cdots-1$ with $|u|=n$. Then if there is a word $v\in\mathcal{L}(X)$ such that $wvu\in\mathcal{L}(X)$. Then there are at least $m(n)+1$ ``0'' in $v$, so $|v|\geq m(n)+1$. Therefore, for any $0<\tilde{\kappa}<\kappa$, $\tilde{m}(\varepsilon,x,y,n)$ with $\lim\limits_{n\to+\infty}\frac{\tilde{m}(\varepsilon,x,y,n)}{n}=\tilde{\kappa}$, $(X,T)$ is not $\tilde{m}$-transitive.

Third, since for any $w\in\mathcal{L}(X)$, $v\in\mathbb{N^{+}}$ with $v\geq1+m(|w|)$, we have that $\cdots w0^vw0^vw0^vw0^v\cdots\in X$. Which means that the periodic point is dense in $X$.

Fourth, since $X$ contains a $S$-gap shift with $S=\{n\in\mathbb{N}\mid n\geq m(1)+1\}$, we have that $X$ has positive topological entropy by \cite[Proposition 2.1]{BG2015}. Then combining with that the periodic point is dense in $X,$ we have $(X,T)$ is not minimal and is not uniquely ergodic.

Finally, give $k_1,k_2,\cdots,k_t\in\mathbb{N^{+}}$, suppose that $10^{k_1}1^{j_1}\cdots0^{k_t}1^{j_t}\in\mathcal{L}(X)$ with the most ``1'', then we have $m(j_i)+1\leq k_i\leq m(j_i+1)$, for any $1\leq i\leq t$. It means that $k_1+k_2+\cdots+k_t\geq m(j_1+j_2+\cdots+j_t)+t$. Then $10^{k_1+k_2+\cdots+k_t}1^{j_1+j_2+\cdots+j_t}\in\mathcal{L}(X)$. Suppose that $f(n)$: $\mathbb{N^{+}}\to\mathbb{N^{+}}$ and $g(n)$: $\mathbb{N^{+}}\to\mathbb{N^{+}}$ with $\lim\limits_{n\to+\infty}\frac{f(n)}{n}<\frac{1}{10}\kappa$ and $\lim\limits_{n\to+\infty}\frac{g(n)}{n}<\frac{1}{10}\kappa$. When $n$ large enough, we have
\begin{enumerate}[(a)]
	\item $n\geq 2g(n)+2$;
	\item $m(n-j-2g(n))\geq g(n)$, where $j=\max\{t\in\mathbb{N^{+}}\mid g(n)+g(n)+1+f(n)+g(n)\geq m(t)+1\}$.
\end{enumerate}
Let $w=1^{n-(g(n)+1)}0^{g(n)+1}\ and\ u=1^n.$
If there are $\tilde{w},v,\tilde{u}\in\mathcal{L}(X)$ such that
\begin{enumerate}
	\item $|w|=|\tilde{w}|$, $w$ and $\tilde{w}$ are $g(n)$ letters different at most;
	\item $|u|=|\tilde{u}|$, $u$ and $\tilde{u}$ are $g(n)$ letters different at most;
	\item the length of $v$ is $f(n)$ at most;
	\item $\tilde{w}v\tilde{u}\in\mathcal{L}(X)$.
\end{enumerate}
Then by (a) and (b), $\tilde{u}$ contains the letter ``0'' and the number of ``0'' is larger than $m(n-j-2g(n))+1$. It is a contradiction for (b). Therefore, the subshift $X$ does not have the approximate product property.\qed

\bigskip

$\mathbf{Acknowledgements}$. We would like to thank P. Varandas for his comments and suggestions for the paper. X. Hou and X. Tian are supported by National Natural Science Foundation of (grant no. 12071082, 11790273).

\end{document}